\documentclass[pdflatex,sn-mathphys-num]{sn-jnl}

\usepackage{amsfonts}
\usepackage{booktabs}
\usepackage{mathtools}
\usepackage{enumitem}
\usepackage{algorithm}
\usepackage{algpseudocode}
\usepackage{pgfplots}
\usepgfplotslibrary{groupplots}
\pgfplotsset{compat=1.18}
\usepackage{cleveref}

\algrenewcommand\algorithmicrequire{\textbf{Input:}}
\algrenewcommand\algorithmicensure{\textbf{Output:}}

\DeclareMathOperator{\diag}{diag}
\DeclareMathOperator{\ran}{ran}
\renewcommand{\Re}{\operatorname{Re}}

\newcommand{\C}{\mathbb{C}}
\newcommand{\R}{\mathbb{R}}
\newcommand{\N}{\mathbb{N}}
\newcommand{\X}{\mathcal{X}}
\newcommand{\dom}{\operatorname{dom}}

\theoremstyle{thmstyleone}
\newtheorem{theorem}{Theorem}[section]
\newtheorem{proposition}[theorem]{Proposition}

\begin{document}

\title[Operator-ADI balanced truncation]{Operator-ADI balanced truncation of delay differential systems}

\author*[1]{\fnm{Timo} \sur{Reis}}
\email{timo.reis@tu-ilmenau.de}

\affil*[1]{\orgname{Technische Universit\"at Ilmenau},
\orgaddress{\city{Ilmenau}, \country{Germany}}}

\abstract{We develop a balanced truncation method for linear delay
differential systems based on a finite-rank alternating-direction implicit
(ADI) iteration for the associated operator Lyapunov equations. The method
acts directly on the infinite-dimensional history-state realization and does
not discretize the delay interval during Gramian approximation or balancing.
The finite-rank reachability and observability factors are represented exactly
in Takenaka--Malmquist coordinates, reducing the required computations to
finite-dimensional linear algebra. These factors yield a finite-dimensional,
delay-free reduced model. For real system data, the implementation remains in
real arithmetic. Under suitable conditions on the ADI shifts and a gap in the
Hankel singular values, we prove convergence in $\mathcal H_\infty$ to the
reduced transfer function obtained by exact balanced truncation of the
infinite-dimensional system. The construction extends to finitely many
discrete delays. Numerical experiments validate the method against independent
reference computations. 
The complete MATLAB implementation used in the experiments is
provided as supplementary material.}


\maketitle

\section{Introduction}
\label{sec:introduction}

Time delays arise through transport, communication, processing, and many
other mechanisms; see, e.g.,
\cite{Erneux2009,InspergerStepan2011,Smith2011,SipahiEtAl2011}.
Although a delay differential equation may be described by finitely many
matrices and delay parameters, its natural state is infinite-dimensional,
since the evolution depends on a~solution history; see,
e.g., \cite{HaleVerduynLunel1993,BatkaiPiazzera2005,
MichielsNiculescu2014}. We consider retarded delay differential equations
\begin{equation}
    \dot{x}(t)
    =
    A_0x(t)+A_1x(t-h)+Bu(t),
    \qquad
    y(t)=Cx(t),
    \label{eq:dde}
\end{equation}
where
$A_0,A_1\in\R^{n\times n}$,
$B\in\R^{n\times m}$,
$C\in\R^{p\times n}$, and $h>0$.
Besides the current value, an initial history
$x(\theta)=\phi_0(\theta)$, $-h\leq\theta\leq0$, has to be prescribed.
Introducing
$\phi(t,\theta)=x(t+\theta)$ gives the equivalent ODE--transport
realization
\begin{equation}\begin{aligned}
    \dot{x}(t)
    &=
    A_0x(t)+A_1\phi(t,-h)+Bu(t),
    \\
    \partial_t\phi(t,\theta)
    &=
    \partial_\theta\phi(t,\theta),
    \qquad -h<\theta<0,
    \\
    \phi(t,0)&=x(t),
    \\
    y(t)&=Cx(t).
\end{aligned}\label{eq:transport-realization}
\end{equation}
We therefore work on the Hilbert state space
$\X=\C^n\times L^2(-h,0;\C^n)$,
with state $(x(t),\phi(t,\cdot))$; see
\cite{HaleVerduynLunel1993,BatkaiPiazzera2005}.

Our aim is to apply balanced truncation directly to this
infinite-dimensio\-nal realiza\-tion and to approximate its input-output behavior
by a finite-dimensional (in particular delay-free) system
\begin{equation}
    \dot{x}_r(t)=A_rx_r(t)+B_ru(t),
    \qquad
    y_r(t)=C_rx_r(t).
    \label{eq:reduced-ode}
\end{equation}
Thus the reduced order $r$ is the dimension of the delay-free surrogate and
is not tied to the dimension $n$ of the instantaneous state. Balanced
truncation is applied to the infinite-dimensional history-state
realization, not to the $n$-dimensional instantaneous state alone.
Consequently, $r$ need not be smaller than $n$.

Many reduction approaches preserve the delay structure, for instance by
interpolation or projection
\cite{BeattieGugercin2009,JarlebringDammMichiels2013,
NaderiLordejaniEtAl2020,BeckerRichter2022}, while others first approximate
the history dynamics by a finite-dimensional system
\cite{InspergerStepan2011,BredaMasetVermiglio2005,
MichielsJarlebringMeerbergen2011}. Delay-free approximations based on
Laguerre, Kautz, and more general shift expansions were studied in
\cite{MakilaPartington1999Laguerre,MakilaPartington1999Shift}.
Here the Gramian approximation and balancing are carried out directly on
the infinite-dimensional history-state realization, without discretizing
the delay state in these steps.

The infinite-dimensional theory of balanced truncation is developed in
\cite{GloverCurtainPartington1988,GuiverOpmeer2014,ReisSelig2014}.
Balanced and Hankel-norm approximations for delay systems were considered
in \cite{GloverLamPartington1986,PartingtonGloverZwartCurtain1988}, and
numerical balanced approximations for special scalar delay transfer
functions in \cite{Lam1999}. We use the finite-rank operator ADI iteration
of \cite{OpmeerReisWollner2013} and exploit the delay structure so that
each shifted operator equation reduces to a finite-dimensional linear
system while the history component is represented exactly.

The convergence rate of ADI depends strongly on the shift parameters. For
the delay generator, the relevant spectral data are the characteristic
roots
\[
    \det\Delta(\lambda)=0,
    \qquad
    \Delta(\lambda)
    :=
    \lambda I-A_0-A_1e^{-\lambda h};
\]
see \cite{HaleVerduynLunel1993,MichielsNiculescu2014}.
We localize rightmost roots by a pseudospectral approximation and refine
them directly for the nonlinear characteristic equation
\cite{BredaMasetVermiglio2005,WuMichiels2012}. The resulting roots are used
only for shift selection; neither the Gramian approximation nor the
balanced truncation itself relies on a discretization of the delay
interval.

The main contributions of this paper are as follows.
\begin{enumerate}[label=(\roman*)]
    \item
    For the matrix delay system \ref{eq:dde}, we turn balanced truncation
of the infinite-dimen\-sional history-state realization into a computational
procedure that represents the delay state exactly throughout the ADI and
balancing steps and produces a finite-dimensional, delay-free ODE.
    \item
    Building on the Takenaka--Malmquist interpretation of finite-rank
    operator ADI in \cite{OpmeerReisWollner2013}, we prove convergence of
    the resulting ADI-based reduced models to the exact
    infinite-dimensional balanced truncation. For the delay realization,
    we directly verify the reachability regularity required by this
    argument. In particular, the reduced transfer functions converge
    in $\mathcal H_\infty$.
    \item
    We develop an exact real-arithmetic realization of the finite-rank
    operator ADI iteration in which real shifts require $n\times n$ systems
    and conjugate shift pairs require real $2n\times2n$ block systems, while
    the history components are represented exactly. The construction
    extends to finitely many discrete retarded delays without increasing
    these solve dimensions.
    \item
    We combine characteristic-root computations with a staged ADI shift
    strategy and goal-oriented stopping and order selection, and compare
    the resulting reduced models with those obtained from an independent
    Chebyshev discretization. A MATLAB implementation of the complete
    procedure used in the numerical experiments is provided as
    Online Resource~1.
    \end{enumerate}

\section{Notation and preliminaries}
\label{sec:preliminaries}

We write
$\R_+=[0,\infty)$ and
$\C_+=\{z\in\C:\Re z>0\}$.
For Hilbert spaces $\mathcal X_1$ and $\mathcal X_2$,
$\mathcal L(\mathcal X_1,\mathcal X_2)$ denotes the space of bounded
linear operators, and $A^*$ denotes the Hilbert-space adjoint of an
operator $A$.
We write $\sigma(A)$ and $\rho(A)$ for the spectrum and resolvent set,
respectively. The nuclear and Hilbert--Schmidt norms are denoted by
$\|\cdot\|_1$ and $\|\cdot\|_{\rm HS}$, respectively.
For finite-dimensional matrices, $\|\cdot\|_2$ and $\|\cdot\|_F$ denote
the spectral and Frobenius norms, respectively, $\otimes$ denotes the
Kronecker product, and $M^\top$ denotes the transpose. For an exponentially
stable system with transfer function $G$, we use
\[
    \|G\|_{\mathcal H_\infty}
    :=
    \sup_{\omega\in\R}\|G(\mathrm i\omega)\|_2.
\]
We consider infinite-dimensional linear systems of the form
\begin{equation}
    \dot z(t)=\mathcal A z(t)+\mathcal B u(t),
    \qquad
    y(t)=\mathcal C z(t),
    \label{eq:abstract-system}
\end{equation}
where $\X$ is a Hilbert space,
$\mathcal A:\dom(\mathcal A)\subset\X\to\X$
generates a strongly continuous semigroup $T=(T(t))_{t\geq0}$,
$\mathcal B\in\mathcal L(\C^m,\X)$, and
$\mathcal C\in\mathcal L(\X,\C^p)$.
Throughout the paper, the semigroup is assumed to be exponentially stable,
i.e., there exist $M_T\geq1$ and $\omega>0$ such that
\[
    \|T(t)\|\leq M_T e^{-\omega t},
    \qquad t\geq0.
\]
The balanced truncation and Lyapunov theory used below is available in
considerably more general settings, including unbounded admissible control
and observation operators. The present assumptions are sufficient for the
delay systems considered here and avoid these additional technicalities;
see, e.g., \cite{GloverCurtainPartington1988,GuiverOpmeer2014,
ReisSelig2014,TucsnakWeiss2009}.

\section{Balanced truncation and operator ADI}
\label{sec:bt}

\subsection{Balanced truncation}
\label{sec:bt:bt}

The infinite-time reachability and observability operators associated with
\eqref{eq:abstract-system} are defined by
$\mathcal R:L^2(\R_+;\C^m)\to\X$, $\mathcal O:
\X\to L^2(\R_+;\C^p)$ with
\[
    \mathcal R u
    =
    \int_0^\infty T(t)\mathcal B u(t)\,dt,
    \qquad
    (\mathcal O z)(t)
    =
    \mathcal C T(t)z.
\]
The corresponding reachability and observability Gramians are
$\mathcal P=\mathcal R\mathcal R^*$ and
$\mathcal Q=\mathcal O^*\mathcal O$. They are, respectively, the unique
bounded, self-adjoint, nonnegative solutions of the operator Lyapunov equations
\[
    \mathcal A\mathcal P
    +\mathcal P\mathcal A^*
    +\mathcal B\mathcal B^*
    =0,
    \qquad
    \mathcal A^*\mathcal Q
    +\mathcal Q\mathcal A
    +\mathcal C^*\mathcal C
    =0,
\]
in the corresponding weak sense; see
\cite[Theorem~5.1.1]{TucsnakWeiss2009}.

Since the input and output spaces are finite-dimensional, exponential
stability together with boundedness of $\mathcal B$ and $\mathcal C$
implies that the reachability and observability operators are
Hilbert--Schmidt. Consequently, $\mathcal P$ and $\mathcal Q$ are nuclear and the
\emph{Hankel operator}
$\mathcal H=\mathcal O\mathcal R$
is nuclear; see \cite[Theorem~4]{CurtainSasane2001}. In particular, if
$\sigma_1\geq\sigma_2\geq\cdots\geq0$
denote the singular values of $\mathcal H$, counted with multiplicity, then
\[\sum_{j=1}^\infty \sigma_j<\infty.\]
The numbers $\sigma_j$ are referred to as the \emph{Hankel singular values}
of \eqref{eq:abstract-system}.
A convenient factor-based construction of balanced realizations was
introduced in the finite-dimensional setting in
\cite{TombsPostlethwaite1987} and extended to infinite-dimensional systems
in \cite{ReisSelig2014}. Let $\mathcal X_P$ and
$\mathcal X_Q$ be Hilbert
spaces and let
$R\in\mathcal L(\mathcal X_P,\X)$,
$S\in\mathcal L(\mathcal X_Q,\X)$
satisfy
$\mathcal P=RR^*$,
$\mathcal Q=SS^*$.
By \cite[Theorem~5.1]{ReisSelig2014}, $S^*R$ is unitarily equivalent to
the Hankel operator on the corresponding effective subspaces, and its
nonzero singular values are precisely the Hankel singular values. For
$r\in\N$ with $\sigma_r>0$, let
$U_r:\C^r\to\mathcal X_Q$ and
$V_r:\C^r\to\mathcal X_P$ map the canonical basis vectors of $\C^r$
to leading left and right singular vectors of $S^*R$, respectively, and set
\[
    \Sigma_r=\operatorname{diag}(\sigma_1,\ldots,\sigma_r).
\]
The associated trial and test operators are
\[
    \mathcal V_r
    =
    RV_r\Sigma_r^{-1/2},
    \qquad
    \mathcal W_r
    =
    SU_r\Sigma_r^{-1/2},
\]
and satisfy
$\mathcal W_r^*\mathcal V_r=I_r$.
Provided that
$\ran\mathcal V_r\subset\dom(\mathcal A)$,
the corresponding reduced matrices are given by
\begin{equation}
    A_r
    =
    \mathcal W_r^*\mathcal A\mathcal V_r,\qquad    B_r
    =
    \mathcal W_r^*\mathcal B,
    \qquad
    C_r
    =
    \mathcal C\mathcal V_r
    =
    \mathcal C RV_r\Sigma_r^{-1/2}.
    \label{eq:bt-r}
\end{equation}
Thus, balanced truncation yields a finite-dimensional delay-free system
of the form \eqref{eq:reduced-ode}.
If the truncation does not split a multiple Hankel singular value, the
reduced system is exponentially stable and its transfer function $G_r$
satisfies the error bound
\begin{equation}
    \|G-G_r\|_{\mathcal H_\infty}
    \leq
    2\sum_{j>r}\sigma_j;
    \label{eq:bt-error-bound}
\end{equation}
see \cite{GloverCurtainPartington1988,GuiverOpmeer2014}.
Nuclearity makes the right-hand side finite and implies convergence to
zero as $r\to\infty$. The sharper bounds in
\cite{GloverCurtainPartington1988,GuiverOpmeer2014} count each distinct
truncated Hankel singular value only once; we use the
multiplicity-counted form \eqref{eq:bt-error-bound} throughout.

\subsection{ADI iteration and Takenaka--Malmquist systems}
\label{sec:bt:adi-tm}

We next recall the ADI iteration for operator Lyapunov
equations and its relation to Takenaka--Malmquist systems following
\cite{OpmeerReisWollner2013}. We restrict the presentation to the
reachability Gramian. The corresponding statements for the observability
Gramian follow by dualization. We use shift parameters in the open right
half-plane, which amounts to reflecting the convention used in
\cite{OpmeerReisWollner2013} at the imaginary axis.

Let $(p_j)_{j\geq1}$ be a sequence in $\C_+$ (referred to as the ADI shift parameters), and set
\[
    \gamma_j=\sqrt{2\Re p_j},
    \qquad j\in\N.
\]
Starting with
\begin{equation}
    V_1=(p_1I-\mathcal A)^{-1}\mathcal B,
    \qquad
    R_1=\gamma_1 V_1,
    \label{eq:adi-first-step}
\end{equation}
the finite-rank ADI iteration is given recursively by
\begin{align}
    V_j
    &=
    V_{j-1}
    -
    (p_j+\overline{p}_{j-1})
    (p_jI-\mathcal A)^{-1}V_{j-1},
    \label{eq:adi-recursion-V}
    \\
    R_j
    &=
    \bigl[
        R_{j-1},\,
        \gamma_jV_j
    \bigr],
    \qquad j\geq2.
    \label{eq:adi-recursion-R}
\end{align}
For finite-dimensional input space $\C^m$, one has
$R_j\in\mathcal L(\C^{mj},\X)$,
and
$\mathcal P_j=R_jR_j^*$
is a finite-rank approximation of the reachability Gramian $\mathcal P$;
cf. \cite[Algorithm~1]{OpmeerReisWollner2013}.

The relation between this iteration and Takenaka--Malmquist systems is
central to its convergence analysis. Associated with the shift sequence
$(p_j)_{j\geq1}$ is the frequency-domain Takenaka--Malmquist system
\[
    \widehat\psi_j(s)
    =
    \frac{\gamma_j}{s+p_j}
    \prod_{\ell=1}^{j-1}
    \frac{s-\overline{p}_\ell}{s+p_\ell},
    \qquad j\geq1.
\]
The inverse Laplace transforms $(\psi_j)_{j\geq1}$ form an orthonormal
system in $L^2(\R_+)$. Writing
\[
    \Psi_j(t):=[\,\psi_1(t)\ \cdots\ \psi_j(t)\,],
\]
they admit the finite-dimensional realization
\begin{equation}
    \Psi_j'(t)=-\Psi_j(t)\mathsf K_j,
    \qquad
    \Psi_j(0)=\mathsf L_j,
    \label{eq:tm-matrix-realization}
\end{equation}
where
\[
    (\mathsf K_j)_{rs}
    =
    \begin{cases}
        p_r, & r=s,\\
        \gamma_r\gamma_s, & r<s,\\
        0, & r>s,
    \end{cases}
    \qquad
    \mathsf L_j=[\,\gamma_1\ \cdots\ \gamma_j\,].
\]
For $j\in\N$, define the isometric embedding
\[
    \iota_j:\C^{mj}\to L^2(\R_+;\C^m),
    \qquad
    \iota_j(v_1,\ldots,v_j)
    =
    \sum_{\ell=1}^j\psi_\ell v_\ell.
\]
A key observation in \cite[Lemma~4.6]{OpmeerReisWollner2013} is that the
ADI factor is precisely the restriction of the reachability operator
$\mathcal R$ to the corresponding Takenaka--Malmquist subspace, that is,
$R_j=\mathcal R\iota_j$.
Consequently,
$\mathcal P_j
    =
    \mathcal R\iota_j\iota_j^*\mathcal R^*$. 
The Takenaka--Malmquist system is complete in $L^2(\R_+)$ if the shifts
satisfy the non-Blaschke condition
\begin{equation}
    \sum_{j=1}^\infty
    \frac{\Re p_j}{1+|p_j|^2}
    =
    \infty.
    \label{eq:non-blaschke}
\end{equation}
Under \eqref{eq:non-blaschke}, the orthogonal projections
$\iota_j\iota_j^*$ converge strongly to the identity and
$\mathcal P_j$ converges monotonically and strongly to $\mathcal P$.
Since $\mathcal R$ is Hilbert--Schmidt, one moreover has
\[\|\mathcal P-\mathcal P_j\|_1
    \rightarrow0,
    \qquad j\to\infty,
\]
where $\|\cdot\|_1$ denotes the nuclear norm; see
\cite[Theorem~4.7]{OpmeerReisWollner2013}.

\subsection{ADI-based balanced truncation}
\label{sec:bt:adi-bt}

We now apply the construction of \Cref{sec:bt:adi-tm} to the reachability
Lyapunov equation and, by dualization, to the observability Lyapunov
equation. Let
\[
    \iota_j:\C^{mj}\to L^2(\R_+;\C^m),
    \qquad
    \jmath_\ell:\C^{p\ell}\to L^2(\R_+;\C^p)
\]
be the isometric embeddings associated with the first $j$ and $\ell$
functions of the respective Takenaka--Malmquist systems. The corresponding
finite-rank ADI factors satisfy
$R_j=\mathcal R\iota_j$,
$S_\ell=\mathcal O^*\jmath_\ell$.
Consequently,
$S_\ell^*R_j
    =
    \jmath_\ell^*\mathcal H\iota_j$.
As is customary in large-scale finite-dimensional balanced truncation, we
replace the exact Gramian factors by the finite-rank factors obtained from
the ADI iteration and apply the square-root balancing procedure to these
approximate factors; see, e.g., \cite{BennerLiPenzl2008}. Thus, let
\[
    S_\ell^*R_j=U\widehat\Sigma V^*,
    \qquad
    S_\ell^*R_j\in\C^{p\ell\times mj},
\]
be a singular value decomposition. The diagonal entries of
$\widehat\Sigma$ are approximations of the Hankel singular values. For a reduced order $r$ with $\widehat\sigma_r>0$, let $U_r$ and $V_r$
contain the corresponding leading $r$ left and right singular vectors,
respectively, and let $\widehat\Sigma_r$ contain the leading $r$ singular
values. We then define
\[
    \mathcal V_r
    =
    R_jV_r\widehat\Sigma_r^{-1/2},
    \qquad
    \mathcal W_r
    =
    S_\ell U_r\widehat\Sigma_r^{-1/2}.
\]
By construction,
$\mathcal W_r^*\mathcal V_r=I_r$. Whenever
$\ran R_j\subset\dom(\mathcal A)$, the ADI-based reduced model is obtained
from these trial and test operators as in \eqref{eq:bt-r}.

In finite dimensions, perturbation and error analysis for balanced
truncation based on approximate Gramians has recently been developed in
\cite{ZhangLi2024}. Related convergence results for infinite-dimensional
approximation schemes can be found in
\cite{Singler2012,GuiverOpmeer2014}. These results do not directly apply here. Indeed, the projections onto the
finite-dimensional Takenaka--Malmquist subspaces need not preserve the
shift structure of the Hankel operator, so that
$\jmath_\ell\jmath_\ell^*\mathcal H\iota_j\iota_j^*$
is in general not itself the Hankel operator of an approximating system.
We therefore work directly with the singular subspaces of these
compressions.
\begin{theorem}[Convergence of ADI-based balanced truncation]
\label{thm:adi-bt-convergence}
Let $\mathcal X$ be a Hilbert space, let
$\mathcal A:\dom(\mathcal A)\subset\mathcal X\to\mathcal X$
generate an exponentially stable strongly continuous semigroup, and let
$\mathcal B\in\mathcal L(\C^m,\mathcal X)$ and
$\mathcal C\in\mathcal L(\mathcal X,\C^p)$. Assume moreover that
$\ran\mathcal R\subset\dom(\mathcal A)$.
Let the reachability and observability shift sequences satisfy the
non-Blaschke condition \eqref{eq:non-blaschke}, and let $r\in\N$ satisfy
$\sigma_r>\sigma_{r+1}$.
Let $G_r^{\rm BT}$ denote the transfer function of the exact order-$r$
balanced truncation obtained from the Gramian factors
$\mathcal R$ and $\mathcal O^*$, and let
$\widehat G_r^{(\ell,j)}$ be the order-$r$ reduced transfer function
obtained from the ADI factors
$R_j=\mathcal R\iota_j$ and
$S_\ell=\mathcal O^*\jmath_\ell$.

Then, for all sufficiently large $j$ and $\ell$,
$\widehat G_r^{(\ell,j)}$ is stable and
\[
    \|\widehat G_r^{(\ell,j)}-G_r^{\rm BT}\|_{\mathcal H_\infty}
    \rightarrow0.
\]
Consequently,
\begin{equation}
    \|G-\widehat G_r^{(\ell,j)}\|_{\mathcal H_\infty}
    \rightarrow
    \|G-G_r^{\rm BT}\|_{\mathcal H_\infty}
    \leq
    2\sum_{k>r}\sigma_k.
    \label{eq:adi-bt-asymptotic-bound}
\end{equation}
\end{theorem}

\begin{proof}
Set
$P_j:=\iota_j\iota_j^*$,
$Q_\ell:=\jmath_\ell\jmath_\ell^*$,
$\mathcal H_{\ell,j}:=Q_\ell\mathcal H P_j$.
The non-Blaschke condition implies
$P_j\to I$ and $Q_\ell\to I$ strongly. Since $\mathcal H$ is nuclear,
\[
    \|\mathcal H_{\ell,j}-\mathcal H\|_1
    \leq
    \|(I-Q_\ell)\mathcal H\|_1
    +
    \|\mathcal H(I-P_j)\|_1
    \rightarrow0.
\]
Hence $\mathcal H_{\ell,j}\to\mathcal H$ in operator norm and, in
particular,
$\widehat\sigma_k^{(\ell,j)}\to\sigma_k$
for every fixed $k$. Thus the gap
$\widehat\sigma_r^{(\ell,j)}>
 \widehat\sigma_{r+1}^{(\ell,j)}$
persists for all sufficiently large $j,\ell$.

Let $E_r,F_r$ denote the leading right and left singular subspaces of
$\mathcal H$, and let
$E_r^{(\ell,j)},F_r^{(\ell,j)}$ denote those of
$\mathcal H_{\ell,j}$. By the gap assumption and standard perturbation
theory for spectral projections,
$E_r^{(\ell,j)}\to E_r$,
$F_r^{(\ell,j)}\to F_r$
in the gap topology; see \cite{Kato1995}. We may therefore choose
isometries
$V_{\ell,j}\to V$,
$U_{\ell,j}\to U$
whose ranges are
$E_r^{(\ell,j)},F_r^{(\ell,j)}$ and $E_r,F_r$, respectively.

Set
\[
    D_{\ell,j}
    :=
    U_{\ell,j}^*\mathcal H V_{\ell,j},
    \qquad
    D:=U^*\mathcal H V.
\]
Since the corresponding singular values are positive,
$D$ is invertible and so is $D_{\ell,j}$ for all sufficiently large
$j,\ell$. Moreover,
$D_{\ell,j}\to D$. Since
$E_r^{(\ell,j)}\subset\ran P_j$,
$F_r^{(\ell,j)}\subset\ran Q_\ell$,
we also have
$D_{\ell,j}
    =
    U_{\ell,j}^*\mathcal H_{\ell,j}V_{\ell,j}$.

Define
$\mathcal V_{\ell,j}
    :=
    \mathcal R V_{\ell,j}$,
$\mathcal W_{\ell,j}
    :=
    \mathcal O^*U_{\ell,j}D_{\ell,j}^{-*}$,
and analogously
$\mathcal V
    :=
    \mathcal R V$,
$\mathcal W
    :=
    \mathcal O^*UD^{-*}$.
Then
$\mathcal W_{\ell,j}^*\mathcal V_{\ell,j}=I_r$, $\mathcal W^*\mathcal V=I_r$.
Their ranges are precisely the trial and test spaces of the corresponding
ADI-based and exact balanced truncations. Hence changing from these bases
to any singular-vector bases changes only the reduced realization, not
its transfer function; cf. \cite[Theorem~2.3]{ZhangLi2024} for the
finite-dimensional counterpart.

Since
$\ran\mathcal R\subset\dom(\mathcal A)$ and $\mathcal A$ is closed, the
closed graph theorem gives
\[
    \mathcal A\mathcal R
    \in
    \mathcal L\bigl(L^2(\R_+;\C^m),\mathcal X\bigr).
\]
Consequently,
$\mathcal V_{\ell,j}\to\mathcal V$,
$\mathcal W_{\ell,j}\to\mathcal W$, and
$\mathcal A\mathcal V_{\ell,j}
    \to
    \mathcal A\mathcal V$.
Thus the finite-dimensional Petrov--Galerkin matrices associated with
these bases converge to those of the exact balanced truncation.

The latter is stable because the truncation does not split a multiple
Hankel singular value. Hence the ADI-based reduced models are stable for
all sufficiently large $j,\ell$. Since the corresponding reduced matrices
have fixed dimension and converge to those of the exact balanced
truncation, the resolvent identity yields uniform convergence on compact
subsets of the imaginary axis. For sufficiently large $|\omega|$, a
Neumann-series estimate gives a uniform $\mathcal O(|\omega|^{-1})$ bound
for the resolvents. Consequently,
\[
    \|\widehat G_r^{(\ell,j)}-G_r^{\rm BT}\|_{\mathcal H_\infty}
    \to0.
\]
Finally,
\[
    \left|
    \|G-\widehat G_r^{(\ell,j)}\|_{\mathcal H_\infty}
    -
    \|G-G_r^{\rm BT}\|_{\mathcal H_\infty}
    \right|
    \leq
    \|\widehat G_r^{(\ell,j)}-G_r^{\rm BT}\|_{\mathcal H_\infty},
\]
and \eqref{eq:bt-error-bound} yields
\eqref{eq:adi-bt-asymptotic-bound}.
\end{proof}

\section{Delay systems}
\label{sec:delay-systems}

The transport realization in \eqref{eq:transport-realization} suggests the
following infinite-dimen\-sio\-nal state-space representation of
\eqref{eq:dde}. We consider \eqref{eq:abstract-system} on
$\X=\C^n\times L^2(-h,0;\C^n)$
with
\begin{equation}
\begin{aligned}
    \mathcal A
    \begin{pmatrix}
        \xi\\ \phi
    \end{pmatrix}
    &=
    \begin{pmatrix}
        A_0\xi+A_1\phi(-h)\\
        \phi'
    \end{pmatrix},
    \qquad     \mathcal Bu=
    \begin{pmatrix}Bu\\0\end{pmatrix},
    \qquad
    \mathcal C
    \begin{pmatrix}\xi\\\phi\end{pmatrix}
    =C\xi,\\
    \dom(\mathcal A)
    &=
    \left\{
        \begin{pmatrix}\xi\\\phi\end{pmatrix}
        \in\C^n\times H^1(-h,0;\C^n)
        :
        \phi(0)=\xi
    \right\},
\end{aligned}    \label{eq:dde-generator}
\end{equation}
The operator $\mathcal A$ generates a strongly continuous semigroup on
$\X$; see, e.g., \cite{Delfour1977,Vinter1978,BernierManitius1978}.
Together with $\mathcal B$ and $\mathcal C$, it yields an
infinite-dimensional realization of the delay differential system
\eqref{eq:dde} in the form \eqref{eq:abstract-system}.

The adjoint of $\mathcal A$ is explicitly characterized in
\cite{Vinter1978}. Since the system matrices are real, it is given by
\begin{align*}
    \mathcal A^*
    \begin{pmatrix}
        \zeta\\ \psi
    \end{pmatrix}
    &=
    \begin{pmatrix}
        A_0^\top\zeta+\psi(0)\\
        -\psi'
    \end{pmatrix},
    \\
    \dom(\mathcal A^*)
    &=
    \left\{
        \begin{pmatrix}\zeta\\\psi\end{pmatrix}
        \in\C^n\times H^1(-h,0;\C^n)
        :
        \psi(-h)=A_1^\top\zeta
    \right\}.
\end{align*}
By \cite[Lemma~2.4.5]{CurtainZwart1995}, the spectrum of
$\mathcal A$ is characterized by the zeros of $\det\Delta(\cdot)$.
Since the system matrices are real,
$\Delta(\overline s)=\overline{\Delta(s)}$, and the characteristic roots
occur in complex conjugate pairs. Together with
$\sigma(\mathcal A^*)=\overline{\sigma(\mathcal A)}$, this yields
\begin{equation}
    \sigma(\mathcal A)
    =
    \sigma(\mathcal A^*)
    =
    \{\lambda\in\C:\det\Delta(\lambda)=0\}.
    \label{eq:dde-spectra}
\end{equation}

For $s\in\rho(\mathcal A)$, the resolvent formula
\cite[Lemma~2.4.5]{CurtainZwart1995} gives
\[
    (sI-\mathcal A)^{-1}
    \begin{pmatrix}f\\g\end{pmatrix}
    =
    \begin{pmatrix}\xi\\\phi\end{pmatrix},
\]
where
\[
    \Delta(s)\xi
    =
    f+
    A_1\int_{-h}^0 e^{s(-h-\tau)}g(\tau)\,d\tau,
    \qquad
    \phi(\theta)
    =
    e^{s\theta}\xi+
    \int_\theta^0 e^{s(\theta-\tau)}g(\tau)\,d\tau.
\]
Likewise,
\[
    (sI-\mathcal A^*)^{-1}
    \begin{pmatrix}f\\g\end{pmatrix}
    =
    \begin{pmatrix}\zeta\\\psi\end{pmatrix}
\]
is characterized by
\[
    \Delta(s)^\top\zeta
    =
    f+\int_{-h}^0e^{s\tau}g(\tau)\,d\tau,
    \qquad
    \psi(\theta)
    =
    e^{-s(\theta+h)}A_1^\top\zeta+
    \int_{-h}^{\theta}e^{-s(\theta-\tau)}g(\tau)\,d\tau.
\]
Hence, for the same shift $s$, the two resolvent evaluations require
linear solves with $\Delta(s)$ and $\Delta(s)^\top$, respectively.

For later use, we recall the classical stability criterion for retarded
delay differential equations.
\begin{proposition}
\label{prop:dde-stability}
The semigroup generated by $\mathcal A$ in \eqref{eq:dde-generator} is
exponentially stable if and only if
\[
    \sup\bigl\{
        \Re\lambda:
        \det\Delta(\lambda)=0
    \bigr\}<0.
\]
\end{proposition}

\begin{proof}
This follows from the characterization of the spectrum in
\eqref{eq:dde-spectra} together with the spectral mapping property for the
retarded delay semigroup; see
\cite[Sec.~7.6, Corollary~6.1, p.~215]{HaleVerduynLunel1993} and
\cite[Ch.~4]{BatkaiPiazzera2005}.
\end{proof}

Sobolev regularity of reachable states for retarded systems is classical;
see \cite{BanksJacobsLangenhop1975,ManitiusTriggiani1978}. For completeness,
we record the short argument needed for the infinite-time reachability
operator used here.

\begin{proposition}[Regularity of the reachability operator]
\label{prop:dde-reachability-regularity}
Let $A_0,A_1\in\R^{n\times n}$, $B\in\R^{n\times m}$, and $h>0$,
and let $\mathcal A$ and $\mathcal B$ be given by
\eqref{eq:dde-generator}. If the semigroup generated by $\mathcal A$ is
exponentially stable, then
$\ran\mathcal R\subset\dom(\mathcal A)$.
\end{proposition}

\begin{proof}
Let $\Phi:\R_+\to\R^{n\times n}$ denote the fundamental matrix,
characterized by $\Phi(0)=I$ and
\[
    \Phi'(t)
    =
    A_0\Phi(t)
    +
    \mathbf 1_{[h,\infty)}(t)A_1\Phi(t-h)
    \quad\text{for almost every }t>0.
\]
See \cite[Secs.~1.5--1.6]{HaleVerduynLunel1993}. The corresponding
fundamental-solution representation of the semigroup gives, for
$\mathcal Ru=(\xi,\phi)$,
\[
    \xi=\int_0^\infty\Phi(t)Bu(t)\,dt,
    \qquad
    \phi(\theta)
    =
    \int_{-\theta}^\infty\Phi(t+\theta)Bu(t)\,dt .
\]
Since the first component of $T(t)\mathcal Bv$ is $\Phi(t)Bv$,
exponential stability implies
$\Phi B\in L^2(\R_+;\R^{n\times m})$.
The defining equation for $\Phi$ then also gives
$\Phi'B\in L^2(\R_+;\R^{n\times m})$.
Moreover, $\phi(0)=\xi$ and, in the weak sense,
\[
    \phi'(\theta)
    =
    Bu(-\theta)
    +
    \int_{-\theta}^\infty
        \Phi'(t+\theta)Bu(t)\,dt .
\]
Hence
\[
    \|\phi'\|_{L^2(-h,0)}
    \leq
    \bigl(\|B\|+\sqrt h\,\|\Phi'B\|_{L^2(\R_+)}\bigr)
    \|u\|_{L^2(\R_+)},
\]
so that $\phi\in H^1(-h,0;\C^n)$ and $\phi(0)=\xi$.
Thus $\mathcal Ru\in\dom(\mathcal A)$.
\end{proof}
Proposition~\ref{prop:dde-reachability-regularity} verifies the additional
regularity assumption
$\ran\mathcal R\subset\dom(\mathcal A)$ required in
\Cref{thm:adi-bt-convergence}. Hence, for the delay realization, the
convergence result holds whenever the reachability and observability shift
sequences satisfy the non-Blaschke condition and
$\sigma_r>\sigma_{r+1}$.

\section{ADI on the delay state space}
\label{sec:adi}

We now derive exact finite-dimensional coordinates for the operator ADI
iteration \eqref{eq:adi-first-step}--\eqref{eq:adi-recursion-R}
on the delay state space. Throughout this section, the equivalent stability
conditions of \Cref{prop:dde-stability} are assumed. Hence the shifted
resolvents for $p\in\C_+$ are well defined, and their finite-dimensional
components require only linear solves with the characteristic matrix
$\Delta$. It remains to represent the history components.

We use the same shift sequence
$(p_i)_{i\geq1}\subset\C_+$ for reachability and observability, since
$\sigma(\mathcal A)=\sigma(\mathcal A^*)$ by
\eqref{eq:dde-spectra}. The sequence is chosen invariant under complex
conjugation, with each nonreal shift adjacent to its conjugate. We first
derive complex coordinates and then transform them to the real block form
used in the implementation.

The natural coordinates for the history components are provided by the same
Takenaka--Malmquist system that underlies the operator ADI iteration. Let
$\psi_1,\psi_2,\ldots$ be the time-domain Takenaka--Malmquist functions from
\Cref{sec:bt:adi-tm} associated with $(p_i)_{i\geq1}$, and set
\[
    \Psi_k(t)
    :=
    \begin{bmatrix}
        \psi_1(t)&\cdots&\psi_k(t)
    \end{bmatrix}.
\]
By analyticity and orthonormality on $L^2(\R_+)$,
$\psi_1,\ldots,\psi_i$ remain linearly independent when restricted to
$[0,h]$.

For $X\in\C^{n\times m}$ and $H\in\C^{ni\times m}$, define
\[
    \mathfrak E_i^P(X,H)u
    :=
    \begin{pmatrix}
        Xu\\[1mm]
        \theta\mapsto
        \bigl(\Psi_i(-\theta)\otimes I_n\bigr)Hu
    \end{pmatrix},
    \qquad u\in\C^m,
\]
and, for $Y\in\C^{n\times p}$ and $H\in\C^{ni\times p}$,
\[
    \mathfrak E_i^Q(Y,H)v
    :=
    \begin{pmatrix}
        Yv\\[1mm]
        \theta\mapsto
        \bigl(\Psi_i(\theta+h)\otimes I_n\bigr)Hv
    \end{pmatrix},
    \qquad v\in\C^p.
\]
The different parametrizations of the history interval reflect the boundary
conditions in the domains of $\mathcal A$ and $\mathcal A^*$,
respectively.

For $i\geq2$, define $a_i\in\C^i$ and
$M_i\in\C^{i\times(i-1)}$ by
\begin{align}
    e^{-p_i t}
    &=
    \Psi_i(t)a_i,
    \label{eq:delay-adi-ai}
    \\
    \int_0^t e^{-p_i(t-s)}\Psi_{i-1}(s)\,ds
    &=
    \Psi_i(t)M_i.
    \label{eq:delay-adi-Mi}
\end{align}
The functions on the left-hand sides belong to
$\operatorname{span}\{\psi_1,\ldots,\psi_i\}$ by the rational structure
of the Takenaka--Malmquist system, and the coefficients are unique by the
linear independence established above. We further set $d_i\in\C^{1\times(i-1)}$ by
\begin{equation}
    d_i
    :=
    \Psi_i(h)M_i
    =
    \int_0^h
        e^{-p_i(h-t)}\Psi_{i-1}(t)\,dt.
    \label{eq:delay-adi-small-data}
\end{equation}

The following result gives the exact representation of the shifted
resolvent evaluations occurring in operator ADI.

\begin{theorem}
\label{thm:delay-resolvent-tm}
Let $A_0,A_1\in\R^{n\times n}$, $h>0$, and let $\mathcal A$ be the
operator in \eqref{eq:dde-generator}. Let $i\geq2$ and
$p_1,\ldots,p_i\in\C_+$ with $p_i\in\rho(\mathcal A)$. Let
$a_i$, $M_i$, and $d_i$ be defined by
\eqref{eq:delay-adi-ai}--\eqref{eq:delay-adi-small-data}.
\begin{enumerate}[label=(\roman*)]
\item
For $X\in\C^{n\times m}$ and
$H\in\C^{n(i-1)\times m}$,
\[
    (p_iI-\mathcal A)^{-1}
    \mathfrak E_{i-1}^P(X,H)
    =
    \mathfrak E_i^P(\widetilde X,\widetilde H),
\]
where
\[    \Delta(p_i)\widetilde X
    =
    X+A_1(d_i\otimes I_n)H,\qquad
    \widetilde H
    =
    (a_i\otimes I_n)\widetilde X
    +(M_i\otimes I_n)H.
\]
\item
For $Y\in\C^{n\times p}$ and
$H\in\C^{n(i-1)\times p}$,
\[
    (p_iI-\mathcal A^*)^{-1}
    \mathfrak E_{i-1}^Q(Y,H)
    =
    \mathfrak E_i^Q(\widetilde Y,\widetilde H),
\]
where
\[\Delta(p_i)^\top\widetilde Y
    =
    Y+(d_i\otimes I_n)H,
    \qquad
    \widetilde H
    =
    (a_i\otimes I_n)A_1^\top\widetilde Y
    +(M_i\otimes I_n)H.
\]
\end{enumerate}
\end{theorem}
\begin{proof}
For (i), the first component of the resolvent formula gives
\[
    \Delta(p_i)\widetilde X
    =
    X+
    A_1
    \int_{-h}^0
        e^{p_i(-h-\tau)}
        \bigl(\Psi_{i-1}(-\tau)\otimes I_n\bigr)H
    \,d\tau.
\]
With $t=-\tau$, the integral equals
$(d_i\otimes I_n)H$. The history component is
\[
    e^{p_i\theta}\widetilde X
    +
    \int_\theta^0
        e^{p_i(\theta-\tau)}
        \bigl(\Psi_{i-1}(-\tau)\otimes I_n\bigr)H
    \,d\tau,
\]
which, after setting $t=-\theta$ and using
\eqref{eq:delay-adi-ai}--\eqref{eq:delay-adi-Mi}, yields the stated expression for
$\widetilde H$.

For (ii), the first component of the adjoint resolvent formula gives
\[
    \Delta(p_i)^\top\widetilde Y
    =
    Y+
    \int_{-h}^0
        e^{p_i\tau}
        \bigl(\Psi_{i-1}(\tau+h)\otimes I_n\bigr)H
    \,d\tau.
\]
The change of variables $t=\tau+h$ turns the integral into
$(d_i\otimes I_n)H$. The corresponding history formula and
\eqref{eq:delay-adi-ai}--\eqref{eq:delay-adi-Mi} then give
\[
    \widetilde H
    =
    (a_i\otimes I_n)A_1^\top\widetilde Y
    +(M_i\otimes I_n)H.
\]
\end{proof}

The first shifted resolvent evaluations are obtained directly from the
resolvent formulas. Namely,
\begin{align*}
    (p_1I-\mathcal A)^{-1}\mathcal B
    &=
    \mathfrak E_1^P
    \bigl(X_1,\gamma_1^{-1}X_1\bigr),
&    \Delta(p_1)X_1&=B,\\
    (p_1I-\mathcal A^*)^{-1}\mathcal C^*
    &=
    \mathfrak E_1^Q
    \bigl(Y_1,\gamma_1^{-1}A_1^\top Y_1\bigr),
    &\Delta(p_1)^\top Y_1&=C^\top.
\end{align*}
Together with \Cref{thm:delay-resolvent-tm}, these formulas give exact
finite-dimensional coordinates for all shifted resolvents in
\eqref{eq:adi-first-step}--\eqref{eq:adi-recursion-V}; when the basis is enlarged,
previous history coefficients are retained by appending a zero block. We
next transform this complex representation into the real block form used in
the implementation.

\subsection{Realification of conjugate shift pairs}
\label{sec:adi-realification}

For real system data, nonreal shifts can be processed in conjugate pairs so
that the ADI iteration remains in real arithmetic, as in finite-dimensional
low-rank ADI \cite{BennerKuerschnerSaak2013}. Here this realification is
performed directly on the Takenaka--Malmquist system.

Suppose that all shifts preceding $p_k$ consist of real shifts and complete
complex conjugate pairs, and let
$p_k=\zeta=a+\mathrm{i}b$,
$p_{k+1}=\overline\zeta$,
$a>0,\quad b>0$.
With $\gamma=\sqrt{2a}$, the corresponding Takenaka--Malmquist functions
satisfy
\[
    \widehat\psi_k(s)
    =
    \frac{\gamma B_{k-1}(s)}{s+\zeta},
    \qquad
    \widehat\psi_{k+1}(s)
    =
    \frac{\gamma B_{k-1}(s)(s-\overline \zeta)}
         {(s+\zeta)(s+\overline \zeta)},
\]
where
\[
    B_{k-1}(s)
    =
    \prod_{\ell=1}^{k-1}
    \frac{s-\overline{p_\ell}}{s+p_\ell}.
\]
By the assumed ordering of the preceding shifts, $B_{k-1}$ has real
coefficients. Set $\rho=\zeta/|\zeta|$ and define
\[
    \mathsf T_\zeta
    :=
    \frac{1}{\sqrt{2}}
    \begin{bmatrix}
        1 & -\rho\\
        1 &  \rho
    \end{bmatrix}.
\]
Then $\mathsf T_\zeta$ is unitary and
\begin{equation}
    \begin{bmatrix}
        \chi_k&\chi_{k+1}
    \end{bmatrix}
    :=
    \begin{bmatrix}
        \psi_k&\psi_{k+1}
    \end{bmatrix}
    \mathsf T_\zeta
    \label{eq:real-tm-transformation}
\end{equation}
is real-valued. Indeed,
\[
    \widehat\chi_k(s)
    =
    \frac{\sqrt{2}\gamma s B_{k-1}(s)}
         {s^2+2as+|\zeta|^2},
    \qquad
    \widehat\chi_{k+1}(s)
    =
    -\frac{\sqrt{2}\gamma |\zeta| B_{k-1}(s)}
          {s^2+2as+|\zeta|^2}.
\]
Since \eqref{eq:real-tm-transformation} is unitary,
$\chi_k$ and $\chi_{k+1}$ are orthonormal and span the same subspace as
$\psi_k$ and $\psi_{k+1}$.

For a real shift, the corresponding Takenaka--Malmquist function is left
unchanged. We then group the shifts into blocks $\pi_1,\ldots,\pi_\nu$, where each block is either a
real shift $p>0$ or a conjugate pair $(\zeta,\overline\zeta)$ with
$\zeta=a+\mathrm{i}b$, $a,b>0$. Associate with each block the real matrices
\begin{equation}
    (K_\ell,L_\ell)
    =
    \begin{cases}
    \left(
        [p],
        [\sqrt{2p}]
    \right),
    & \pi_\ell=p>0,
    \\[3mm]
    \left(
        \begin{bmatrix}
            2a & |\zeta|\\
            -|\zeta| & 0
        \end{bmatrix},
        \begin{bmatrix}
            2\sqrt a&0
        \end{bmatrix}
    \right),
    & \pi_\ell=(\zeta,\overline\zeta).
    \end{cases}
    \label{eq:real-tm-blocks}
\end{equation}
Thus $K_\ell$ has dimension one for a real shift and dimension two for a
conjugate pair.

Let $\delta_\ell$ denote the dimension of the $\ell$th block and
$q_\ell=\delta_1+\cdots+\delta_\ell$. Recursively define
\begin{equation}
    K^{[\ell]}
    =
    \begin{bmatrix}
        K^{[\ell-1]}
        &
        (L^{[\ell-1]})^\top L_\ell
        \\
        0
        &
        K_\ell
    \end{bmatrix},
    \qquad
    L^{[\ell]}
    =
    \begin{bmatrix}
        L^{[\ell-1]}&L_\ell
    \end{bmatrix}.
    \label{eq:real-tm-global-matrices}
\end{equation}
For the first block, $K^{[1]}=K_1$ and $L^{[1]}=L_1$.

\begin{proposition}
\label{prop:real-tm-system}
Let the shift sequence consist of real shifts and complete conjugate pairs,
ordered blockwise as above. Then the Takenaka--Malmquist subspace generated
by the first $q_\ell$ shifts admits a real orthonormal basis
\[
    \Theta^{[\ell]}(t)
    =
    \begin{bmatrix}
        \chi_1(t)&\cdots&\chi_{q_\ell}(t)
    \end{bmatrix}
\]
satisfying
\begin{equation}
    \frac{d}{dt}\Theta^{[\ell]}(t)
    =
    -\Theta^{[\ell]}(t)K^{[\ell]},
    \qquad
    \Theta^{[\ell]}(0)=L^{[\ell]}.
    \label{eq:real-tm-dynamics}
\end{equation}
In particular,
\[
    \Theta^{[\ell]}(t)
    =
    L^{[\ell]}e^{-tK^{[\ell]}}.
\]
The subspace spanned by $\Theta^{[\ell]}$ coincides with that generated by
the original complex Takenaka--Malmquist functions associated with the
same shifts.
\end{proposition}

\begin{proof}
For a real shift there is nothing to show. For a conjugate pair,
\eqref{eq:real-tm-transformation} is a unitary change of basis and produces
two real-valued orthonormal functions. Applying these transformations blockwise to
\eqref{eq:tm-matrix-realization} yields
\eqref{eq:real-tm-global-matrices}--\eqref{eq:real-tm-dynamics}. The diagonal block
associated with a conjugate pair follows from
\[
    \mathsf T_\zeta^*
    \begin{bmatrix}
        \zeta&2a\\
        0&\overline\zeta
    \end{bmatrix}
    \mathsf T_\zeta
    =
    \begin{bmatrix}
        2a&|\zeta|\\
        -|\zeta|&0
    \end{bmatrix}
\]
whereas
\[
    \begin{bmatrix}\gamma&\gamma\end{bmatrix}\mathsf T_\zeta
    =
    \begin{bmatrix}\sqrt{2}\gamma&0\end{bmatrix}.
\]
The off-diagonal blocks transform into
$(L^{[\ell-1]})^\top L_\ell$, which gives
\eqref{eq:real-tm-global-matrices}.
\end{proof}

For a conjugate pair, the matrix exponential required below is itself
available entirely in real arithmetic. If
$K_\ell$ is the $2\times2$ block in
\eqref{eq:real-tm-blocks}, then
\[
    e^{-hK_\ell}
    =
    e^{-ah}
    \left(
        \cos(bh)I_2
        -
        \frac{\sin(bh)}{b}
        (K_\ell-aI_2)
    \right).
\]
Thus no complex arithmetic is required after the shifts have been grouped
into real blocks.

\subsection{Real block ADI}
\label{sec:adi-real-block}

We now use the real Takenaka--Malmquist basis from
\Cref{sec:adi-realification} to obtain fully real implementations of the
reachability and observability ADI iterations. Since the change from the
complex Takenaka--Malmquist basis to $\Theta^{[\ell]}$ is unitary, the
finite-rank Gramian approximations are unchanged.

For the reachability iteration, set
$K_P^{[\ell]}=K^{[\ell]}\otimes I_m$,
$L_P^{[\ell]}=L^{[\ell]}\otimes I_m$.
The corresponding ADI factor has the representation
\begin{equation}
    R^{[\ell]}v
    =
    \begin{pmatrix}
        X^{[\ell]}v\\
        \theta\mapsto
        X^{[\ell]}e^{\theta K_P^{[\ell]}}v
    \end{pmatrix},
    \qquad
    v\in\C^{mq_\ell},
    \label{eq:real-reachability-factor}
\end{equation}
where $X^{[\ell]}\in\R^{n\times mq_\ell}$ satisfies
\[
    A_0X^{[\ell]}
    +
    A_1X^{[\ell]}e^{-hK_P^{[\ell]}}
    -
    X^{[\ell]}K_P^{[\ell]}
    =
    -BL_P^{[\ell]}.
\]
Likewise, with
$K_Q^{[\ell]}=K^{[\ell]}\otimes I_p$,
$L_Q^{[\ell]}=L^{[\ell]}\otimes I_p$,
the observability factor is
\begin{equation}
    S^{[\ell]}v
    =
    \begin{pmatrix}
        Y^{[\ell]}v\\
        \theta\mapsto
        A_1^\top Y^{[\ell]}
        e^{-(\theta+h)K_Q^{[\ell]}}v
    \end{pmatrix},
    \qquad
    v\in\C^{pq_\ell},
    \label{eq:real-observability-factor}
\end{equation}
where
\[
    A_0^\top Y^{[\ell]}
    +
    A_1^\top Y^{[\ell]}e^{-hK_Q^{[\ell]}}
    -
    Y^{[\ell]}K_Q^{[\ell]}
    =
    -C^\top L_Q^{[\ell]}.
\]
Equivalently, the factor representations satisfy
\begin{equation}
    \mathcal A R^{[\ell]}
    =
    R^{[\ell]}K_P^{[\ell]}-\mathcal B L_P^{[\ell]},
    \qquad
    \mathcal A^* S^{[\ell]}
    =
    S^{[\ell]}K_Q^{[\ell]}-\mathcal C^* L_Q^{[\ell]}.
    \label{eq:real-block-operator-identities}
\end{equation}
These identities follow directly from
\eqref{eq:real-reachability-factor}--\eqref{eq:real-observability-factor} and the
two matrix equations above; the boundary conditions at $\theta=0$ and
$\theta=-h$, respectively, are satisfied by construction. Thus no
approximation beyond the finite-rank ADI truncation has been introduced.

The equations can be solved blockwise. Set
$G_\ell:=(L^{[\ell-1]})^\top L_\ell$
and partition
\[
    e^{-hK^{[\ell]}}
    =
    \begin{bmatrix}
        e^{-hK^{[\ell-1]}}&F_\ell\\
        0&E_\ell
    \end{bmatrix},
    \qquad
    E_\ell=e^{-hK_\ell}.
\]
For $\ell=1$, all terms involving the preceding blocks are omitted.

Writing
\[
    X^{[\ell]}
    =
    \begin{bmatrix}
        X^{[\ell-1]}&X_\ell
    \end{bmatrix},
    \qquad
    Y^{[\ell]}
    =
    \begin{bmatrix}
        Y^{[\ell-1]}&Y_\ell
    \end{bmatrix},
\]
with
$X_\ell\in\R^{n\times m\delta_\ell}$ and
$Y_\ell\in\R^{n\times p\delta_\ell}$, the new reachability block is determined by
\begin{align}
    A_0X_\ell
    +A_1X_\ell(E_\ell\otimes I_m)
    -X_\ell(K_\ell\otimes I_m)
    &=
    -B(L_\ell\otimes I_m)
    \notag\\
    &\quad
    -A_1X^{[\ell-1]}(F_\ell\otimes I_m)
    +X^{[\ell-1]}(G_\ell\otimes I_m),
    \label{eq:real-block-reachability-solve}
\end{align}
whereas the new observability block satisfies
\begin{align*}
    A_0^\top Y_\ell
    +A_1^\top Y_\ell(E_\ell\otimes I_p)
    -Y_\ell(K_\ell\otimes I_p)
    &=
    -C^\top(L_\ell\otimes I_p)
    \notag\\
    &\quad
    -A_1^\top Y^{[\ell-1]}(F_\ell\otimes I_p)
    +Y^{[\ell-1]}(G_\ell\otimes I_p).
\end{align*}
Denote the right-hand sides of the reachability and observability block
equations above by $F_\ell^P$ and $F_\ell^Q$, respectively. For a real
shift, \eqref{eq:real-block-reachability-solve} then reduces to
\[
    \Delta(p)X_\ell=-F_\ell^P,
\]
and analogously to
$\Delta(p)^\top Y_\ell=-F_\ell^Q$ for observability.
For a conjugate pair, write
$E_\ell=[e_{rs}]_{r,s=1}^2$ and
$K_\ell=[k_{rs}]_{r,s=1}^2$. Then an equation of the form
\[
        A_0Z+A_1ZE_\ell-ZK_\ell=F,
    \qquad
    Z=\begin{bmatrix}Z_1&Z_2\end{bmatrix},
    \quad
    F=\begin{bmatrix}F_1&F_2\end{bmatrix},
\]
is equivalent to the real coupled system
\begin{equation}
    \begin{bmatrix}
        A_0+e_{11}A_1-k_{11}I_n
        &
        e_{21}A_1-k_{21}I_n
        \\
        e_{12}A_1-k_{12}I_n
        &
        A_0+e_{22}A_1-k_{22}I_n
    \end{bmatrix}
    \begin{bmatrix}
        Z_1\\Z_2
    \end{bmatrix}
    =
    \begin{bmatrix}
        F_1\\F_2
    \end{bmatrix}.
    \label{eq:real-pair-linear-system}
\end{equation}
Hence a conjugate pair is processed by a single real $2n\times2n$ block
system. The observability equation has the same form with
$A_0,A_1$ replaced by $A_0^\top,A_1^\top$.

The resulting blockwise procedures are summarized below.

\begin{algorithm}
\caption{Reachability block ADI}
\label{alg:real-adi-P}
\begin{algorithmic}[1]
\Require $A_0,A_1,B,h$ and shift blocks
$\pi_1,\ldots,\pi_\nu$
\Ensure Real factor data $X^{[\nu]}$, $K^{[\nu]}$, $L^{[\nu]}$

\State $X^{[0]}\gets[\,]$, $K^{[0]}\gets[\,]$, $L^{[0]}\gets[\,]$

\For{$\ell=1,\ldots,\nu$}

    \If{$\pi_\ell=p>0$}
        \State $K_\ell\gets[p]$,
        $L_\ell\gets[\sqrt{2p}]$
    \Else
        \State Write $\pi_\ell=(\zeta,\overline\zeta)$ with
        $\zeta=a+\mathrm{i}b$, $a,b>0$
        \State $K_\ell\gets
        \begin{bmatrix}2a&|\zeta|\\-|\zeta|&0\end{bmatrix}$,
        $L_\ell\gets
        [\,2\sqrt a\;\;0\,]$
    \EndIf

    \State $G_\ell\gets(L^{[\ell-1]})^\top L_\ell$

    \State Assemble
    $K^{[\ell]}\gets
    \begin{bmatrix}
        K^{[\ell-1]}&G_\ell\\
        0&K_\ell
    \end{bmatrix}$
    and
    $L^{[\ell]}\gets[\,L^{[\ell-1]}\;\;L_\ell\,]$

    \State Obtain $F_\ell,E_\ell$ from
    $e^{-hK^{[\ell]}}
    =
    \begin{bmatrix}
        e^{-hK^{[\ell-1]}}&F_\ell\\
        0&E_\ell
    \end{bmatrix}$

    \State $F_\ell^P\gets
    -B(L_\ell\otimes I_m)
    -A_1X^{[\ell-1]}(F_\ell\otimes I_m)
    +X^{[\ell-1]}(G_\ell\otimes I_m)$

    \If{$\pi_\ell=p>0$}
        \State Solve $\Delta(p)X_\ell=-F_\ell^P$
    \Else
        \State Solve \eqref{eq:real-pair-linear-system} with
        $F=F_\ell^P$ for
        $X_\ell=[\,X_{\ell,1}\;\;X_{\ell,2}\,]$
    \EndIf

    \State $X^{[\ell]}\gets[\,X^{[\ell-1]}\;\;X_\ell\,]$

\EndFor

\State \Return $X^{[\nu]}$, $K^{[\nu]}$, $L^{[\nu]}$
\end{algorithmic}
\end{algorithm}

\begin{algorithm}
\caption{Observability block ADI}
\label{alg:real-adi-Q}
\begin{algorithmic}[1]
\Require $A_0,A_1,C,h$ and shift blocks
$\pi_1,\ldots,\pi_\nu$
\Ensure Real factor data $Y^{[\nu]}$, $K^{[\nu]}$, $L^{[\nu]}$

\State $Y^{[0]}\gets[\,]$, $K^{[0]}\gets[\,]$, $L^{[0]}\gets[\,]$

\For{$\ell=1,\ldots,\nu$}

    \If{$\pi_\ell=p>0$}
        \State $K_\ell\gets[p]$,
        $L_\ell\gets[\sqrt{2p}]$
    \Else
        \State Write $\pi_\ell=(\zeta,\overline\zeta)$ with
        $\zeta=a+\mathrm{i}b$, $a,b>0$
        \State $K_\ell\gets
        \begin{bmatrix}2a&|\zeta|\\-|\zeta|&0\end{bmatrix}$,
        $L_\ell\gets
        [\,2\sqrt a\;\;0\,]$
    \EndIf

    \State $G_\ell\gets(L^{[\ell-1]})^\top L_\ell$

    \State Assemble
    $K^{[\ell]}\gets
    \begin{bmatrix}
        K^{[\ell-1]}&G_\ell\\
        0&K_\ell
    \end{bmatrix}$
    and
    $L^{[\ell]}\gets[\,L^{[\ell-1]}\;\;L_\ell\,]$

    \State Obtain $F_\ell,E_\ell$ from
    $e^{-hK^{[\ell]}}
    =
    \begin{bmatrix}
        e^{-hK^{[\ell-1]}}&F_\ell\\
        0&E_\ell
    \end{bmatrix}$

    \State $F_\ell^Q\gets
    -C^\top(L_\ell\otimes I_p)
    -A_1^\top Y^{[\ell-1]}(F_\ell\otimes I_p)
    +Y^{[\ell-1]}(G_\ell\otimes I_p)$

    \If{$\pi_\ell=p>0$}
        \State Solve $\Delta(p)^\top Y_\ell=-F_\ell^Q$
    \Else
        \State Solve \eqref{eq:real-pair-linear-system} with
        $A_0,A_1,F$ replaced by
        $A_0^\top,A_1^\top,F_\ell^Q$ for
        $Y_\ell=[\,Y_{\ell,1}\;\;Y_{\ell,2}\,]$
    \EndIf

    \State $Y^{[\ell]}\gets[\,Y^{[\ell-1]}\;\;Y_\ell\,]$

\EndFor

\State \Return $Y^{[\nu]}$, $K^{[\nu]}$, $L^{[\nu]}$
\end{algorithmic}
\end{algorithm}

The preceding construction yields an exact real-arithmetic realization of
the operator ADI iteration.

\begin{theorem}
\label{thm:real-block-adi}
Assume that the delay system \eqref{eq:dde} is exponentially stable and
that the ADI shift sequence is invariant under complex conjugation and
ordered in real shifts and complete conjugate pairs. Let
$R_{q_\ell}$ and $S_{q_\ell}$ denote the reachability and observability
factors obtained after the first $q_\ell$ scalar shifts by the complex
operator ADI iteration, and let $R^{[\ell]}$ and $S^{[\ell]}$ be the
factors generated by the real block construction above.

Then there exist unitary matrices
$U_P^{[\ell]}\in\C^{mq_\ell\times mq_\ell}$,
$U_Q^{[\ell]}\in\C^{pq_\ell\times pq_\ell}$
such that
$R^{[\ell]}
    =
    R_{q_\ell}U_P^{[\ell]}$,
$S^{[\ell]}
    =
    S_{q_\ell}U_Q^{[\ell]}$.
Consequently,
$R^{[\ell]}(R^{[\ell]})^*
    =
    R_{q_\ell}R_{q_\ell}^*$,
$S^{[\ell]}(S^{[\ell]})^*
    =
    S_{q_\ell}S_{q_\ell}^*$.
Thus the real block iteration produces exactly the same finite-rank
reachability and observability Gramian approximations as the complex
operator ADI iteration. Each real shift requires an $n\times n$ real
linear solve, whereas each nonreal conjugate pair requires a single real
$2n\times2n$ coupled solve.
\end{theorem}

\begin{proof}
By \Cref{prop:real-tm-system}, the real Takenaka--Malmquist basis is
obtained from the complex one by a blockwise unitary change of basis. Tensoring these
transformations with $I_m$ and $I_p$ yields the unitary matrices
$U_P^{[\ell]}$ and $U_Q^{[\ell]}$, respectively. Since the ADI factors
are the restrictions of the reachability and observability operators to
the corresponding Takenaka--Malmquist subspaces, the stated factor
relations follow. The identities for the Gramian approximations are then
immediate. The dimensions of the required real linear systems follow from
\eqref{eq:real-block-reachability-solve}--\eqref{eq:real-pair-linear-system} and
the analogous observability equations.
\end{proof}

\subsection{Computational cost and storage}
\label{sec:adi-complexity}

Let $\nu_1$ and $\nu_2$ denote the numbers of real and conjugate-pair
shift blocks, respectively, and let
$q=\nu_1+2\nu_2$
be the total number of scalar shifts. For a prescribed shift sequence, the
large-scale computational work of the ADI iteration is concentrated in the
linear systems in
\Cref{alg:real-adi-P,alg:real-adi-Q}. A real shift requires primal and
dual solves of order $n$ with $m$ and $p$ right-hand sides, respectively,
whereas a conjugate pair requires the corresponding real coupled systems
of order $2n$. Repeated shift blocks lead to identical coefficient
matrices, so their factorizations may in principle be reused.

Apart from operations involving the factor arrays $X$ and $Y$, the
remaining dense computations associated with the shift representation
are performed in the ADI dimension $q$. After $q$ scalar shifts, the
factors are represented by
$X\in\R^{n\times mq}$,
$Y\in\R^{n\times pq}$,
$K\in\R^{q\times q}$,
$L\in\R^{1\times q}$.
Their storage therefore requires
\[
    \mathcal O\bigl(nq(m+p)+q^2\bigr)
\]
numbers, apart from the storage required by the linear solver. In
particular, no discretization of the delay interval is stored.

Since $K^{[\ell]}$ is block upper triangular,
\[
    e^{-hK^{[\ell]}}
    =
    \begin{bmatrix}
        e^{-hK^{[\ell-1]}} & F_\ell\\
        0 & e^{-hK_\ell}
    \end{bmatrix},
\]
the required matrix-exponential updates are carried out in the
$q$-dimensional shift space. Their cost therefore depends on the number of
accumulated shifts $q$, but not on the dimension $n$ of the instantaneous
state.

\subsection{Shift parameter selection}
\label{sec:adi-shifts}

The non-Blaschke condition \eqref{eq:non-blaschke} guarantees asymptotic
convergence but says nothing about its rate. As in finite-dimensional
Lyapunov ADI, effective shifts should reflect spectral information; see,
e.g., \cite{Penzl2000,Wachspress2013,BennerKuerschnerSaak2014}. For the
delay generator this information is provided by the characteristic roots,
the zeros of $\det\Delta(\cdot)$.

We first obtain candidate approximations to rightmost characteristic roots
from a finite-order Chebyshev pseudospectral discretization of the generator
\cite{BredaMasetVermiglio2005,WuMichiels2012}. The resulting eigenvalues
are used only as initial guesses and are refined directly for the nonlinear
characteristic equation, following \cite{WuMichiels2012}. If $z$ is an approximate root and $v,w\in\C^n$ are unit right and left
singular vectors corresponding to the smallest singular value of
$\Delta(z)$, we apply the local correction
\[
    z_{\rm new}
    =
    z-
    \frac{w^*\Delta(z)v}{w^*\Delta'(z)v},
    \qquad
    \Delta'(z)=I+hA_1e^{-zh},
\]
provided that the denominator is nonzero, and iterate until
$\sigma_{\min}(\Delta(z))$ is sufficiently small.

Let $\lambda_1,\ldots,\lambda_{N_\lambda}$ be the retained refined root
candidates, ordered from right to left and closed under complex
conjugation. A scalar ADI shift
is obtained by reflection at the imaginary axis,
$p=-\overline\lambda\in\C_+$.
Real roots give scalar shift blocks, while a conjugate pair is represented
by $\pi=(\zeta,\overline\zeta)$ with
$\zeta=-\overline\lambda$ and $\Im\zeta>0$, as in
\Cref{sec:adi-realification}. The same blocks are used for reachability and
observability, since $\mathcal A$ and $\mathcal A^*$ have the same spectrum
by \eqref{eq:dde-spectra}.

A fixed cyclic shift sequence satisfies the non-Blaschke condition, but may
converge slowly after its spectral region has been resolved. We therefore
process successive bands of characteristic roots, starting from the
rightmost ones and enlarging the spectral range as the iteration proceeds.
If the available shift blocks are exhausted before convergence, the
pseudospectral search region is enlarged and further candidates are
localized and refined.
For the asymptotic statements, any finite staged prefix can be continued by
cyclically repeating a nonempty set of available blocks. This leaves the
computed prefix unchanged and yields an infinite shift sequence satisfying
the non-Blaschke condition.

\section{Balanced truncation without delay discretization}
\label{sec:realization}

We now combine the real ADI factors of
\Cref{thm:real-block-adi} with square-root balanced truncation. The
quantities
$(S^{[\nu]})^*R^{[\nu]}$,
$(S^{[\nu]})^*\mathcal A R^{[\nu]}$,
$(S^{[\nu]})^*\mathcal B$,
$\mathcal C R^{[\nu]}$
are obtained directly from the finite-dimensional data returned by
\Cref{alg:real-adi-P,alg:real-adi-Q}, without discretizing the delay
interval.

\subsection{Computation of the reduced model}
\label{sec:realization-rom}

Suppose that $\nu$ shift blocks have been processed and let
$q=q_\nu$ denote the total number of corresponding
Takenaka--Malmquist functions. Since the same shift blocks are used for
the reachability and observability iterations, the matrices
$K:=K^{[\nu]}\in\R^{q\times q}$,
$L:=L^{[\nu]}\in\R^{1\times q}$
are common to both factors. Let
$X:=X^{[\nu]}\in\R^{n\times mq}$,
$Y:=Y^{[\nu]}\in\R^{n\times pq}$
be the coefficient matrices returned by
\Cref{alg:real-adi-P,alg:real-adi-Q}, respectively. All shift information
required for the postprocessing is encoded in the common matrices $K$ and
$L$. Set
$K_P:=K\otimes I_m$,
$K_Q:=K\otimes I_p$,
$L_P:=L\otimes I_m$.
With $\ell=\nu$, the factor representations
\eqref{eq:real-reachability-factor}--\eqref{eq:real-observability-factor}
take the form
\[
    R^{[\nu]}v
    =
    \begin{pmatrix}
        Xv\\
        \theta\mapsto Xe^{\theta K_P}v
    \end{pmatrix},
    \qquad
    S^{[\nu]}w
    =
    \begin{pmatrix}
        Yw\\
        \theta\mapsto
        A_1^\top Ye^{-(\theta+h)K_Q}w
    \end{pmatrix}.
\]
We first compute the cross product $(S^{[\nu]})^*R^{[\nu]}$. Set
$D:=Y^\top A_1X
    \in\R^{pq\times mq}$.
Using these representations, the contribution of the history components is
\[
    \Xi_h
    =
    \int_0^h
        e^{-tK_Q^\top}
        D
        e^{-(h-t)K_P}\,dt.
\]
For $\alpha=1,\ldots,p$ and $\beta=1,\ldots,m$, define
$D^{\alpha\beta}\in\R^{q\times q}$ by
\begin{equation}
    (D^{\alpha\beta})_{ij}
    :=
    D_{(i-1)p+\alpha,\,(j-1)m+\beta},
    \qquad i,j=1,\ldots,q.
    \label{eq:channel-block-D}
\end{equation}
The standard block-exponential construction
\cite{VanLoan1978} gives
\begin{equation}
    \exp\left(
        h
        \begin{bmatrix}
            -K^\top&D^{\alpha\beta}\\
            0&-K
        \end{bmatrix}
    \right)
    =
    \begin{bmatrix}
        e^{-hK^\top}&\Xi^{\alpha\beta}\\
        0&e^{-hK}
    \end{bmatrix},
    \label{eq:channel-block-exponential}
\end{equation}
where each block has size $q\times q$, and $\Xi_h$ is assembled according to
\begin{equation}
    (\Xi_h)_{(i-1)p+\alpha,\,(j-1)m+\beta}
    =
    (\Xi^{\alpha\beta})_{ij}.
    \label{eq:assemble-Xi}
\end{equation}
Consequently,
\[
    M
    :=
    (S^{[\nu]})^*R^{[\nu]}
    =
    Y^\top X+\Xi_h.
\]
The remaining products follow directly from the state components of the
two factors. In particular,
\[
    B_{\rm f}
    :=
    (S^{[\nu]})^*\mathcal B
    =
    Y^\top B,
    \qquad
    C_{\rm f}
    :=
    \mathcal C R^{[\nu]}
    =
    CX.
\]
Using \eqref{eq:real-block-operator-identities},
\[
    F
    :=
    (S^{[\nu]})^*\mathcal A R^{[\nu]}
    =
    MK_P-B_{\rm f}L_P.
\]
Let
$M
    =
    U\widehat\Sigma V^\top$
be a singular value decomposition, and let
$U_r$, $V_r$, and
\[
    \widehat\Sigma_r
    =
    \diag(\widehat\sigma_1,\ldots,\widehat\sigma_r)
\]
contain the leading $r$ singular triplets. The singular values $\widehat\sigma_i$ approximate the Hankel singular
values of the original system. Inserting these factors into the square-root
balancing formulas of \Cref{sec:bt:adi-bt} gives
\[
    A_r
    =
    \widehat\Sigma_r^{-1/2}
    U_r^\top F V_r
    \widehat\Sigma_r^{-1/2},
    \qquad
    B_r
    =
    \widehat\Sigma_r^{-1/2}
    U_r^\top B_{\rm f},
    \qquad
    C_r
    =
    C_{\rm f}V_r
    \widehat\Sigma_r^{-1/2}.
\]
Hence these matrices
define the reduced model \eqref{eq:reduced-ode} entirely from the real
finite-dimensional data generated by the block ADI iteration. Neither a
discretization of the delay interval nor an explicit construction of the
infinite-dimensional balancing transformations is required.

The complete postprocessing step is summarized in
\Cref{alg:adi-to-rom}.

\begin{algorithm}
\caption{Reduced model from real block ADI data}
\label{alg:adi-to-rom}
\begin{algorithmic}[1]
\Require $A_1,B,C,h$, block-ADI data $X,Y,K,L$, and reduced order $r$
\Ensure $A_r,B_r,C_r$, approximate Hankel singular values, and
$\widehat\varepsilon_r$

\State $K_P\gets K\otimes I_m$, $L_P\gets L\otimes I_m$
\State $D\gets Y^\top A_1X$

\For{$\alpha=1,\ldots,p$}
    \For{$\beta=1,\ldots,m$}
        \State Form $D^{\alpha\beta}$ by \eqref{eq:channel-block-D} and
        compute $\Xi^{\alpha\beta}$ from
        \eqref{eq:channel-block-exponential}
    \EndFor
\EndFor

\State Assemble $\Xi_h$ according to \eqref{eq:assemble-Xi}
\State $M\gets Y^\top X+\Xi_h$,
$B_{\rm f}\gets Y^\top B$, $C_{\rm f}\gets CX$
\State $F\gets MK_P-B_{\rm f}L_P$

\State Compute $M=U\widehat\Sigma V^\top$ and let
$\widehat\sigma_1\ge\cdots\ge\widehat\sigma_s\ge0$,
$s=q\min\{m,p\}$

\State Check $1\le r\le s$ and $\widehat\sigma_r>0$

\State $\widehat\varepsilon_r\gets
2\sum_{k=r+1}^{s}\widehat\sigma_k$

\State Retain $U_r,V_r$ and
$\widehat\Sigma_r=\diag(\widehat\sigma_1,\ldots,\widehat\sigma_r)$

\State $A_r\gets
\widehat\Sigma_r^{-1/2}
U_r^\top F V_r
\widehat\Sigma_r^{-1/2}$

\State $B_r\gets
\widehat\Sigma_r^{-1/2}U_r^\top B_{\rm f}$,
$C_r\gets
C_{\rm f}V_r\widehat\Sigma_r^{-1/2}$

\State \Return $A_r,B_r,C_r$,
$\widehat\sigma_1,\ldots,\widehat\sigma_s$,
$\widehat\varepsilon_r$
\end{algorithmic}
\end{algorithm}

\subsection{Computational cost of the balancing step}
\label{sec:balancing-complexity}

After completion of the ADI iteration, the balancing step operates on
\[
    X\in\R^{n\times mq},
    \qquad
    Y\in\R^{n\times pq},
    \qquad
    M=(S^{[\nu]})^*R^{[\nu]}
      \in\R^{pq\times mq}.
\]
In dense arithmetic, applying $A_1$ to $X$ costs
$\mathcal O(n^2mq)$ operations, while forming
$Y^\top X$ and $Y^\top A_1X$ requires
$\mathcal O(nmpq^2)$ operations. For the history contribution to $M$,
the block-exponential formula of \cite{VanLoan1978} requires one matrix
exponential of order $2q$ for each of the $mp$ input-output channel pairs,
resulting in
$\mathcal O(mpq^3)$ operations.

Using the Kronecker structure of $K_P$, forming
$F=MK_P-B_{\rm f}L_P$
requires $\mathcal O(mpq^3)$ operations up to lower-order terms. A dense
singular value decomposition of $M\in\R^{pq\times mq}$ costs
$\mathcal O\bigl(mp\min\{m,p\}q^3\bigr)$.
Once the leading $r$ singular triplets have been computed, the reduced
matrices can be formed in
$\mathcal O\bigl(
        rmpq^2+\min\{m,p\}qr^2
    \bigr)$
operations by using the more favorable multiplication order.

Thus the overall cost of the balancing step is
\[\mathcal O\bigl(
        n^2mq
        +nmpq^2
        +mpq^3
        +mp\min\{m,p\}q^3
        +rmpq^2
        +\min\{m,p\}qr^2
    \bigr).\]
For small input and output dimensions, the only operations in this
postprocessing step involving the large state dimension $n$ are therefore
the application of $A_1$ and the cross products with $X$ and $Y$.
\subsection{Stopping and order selection}
\label{sec:stopping}

The non-Blaschke condition is asymptotic and therefore does not provide a
practical stopping rule. We use goal-oriented criteria based on nested
reduced models. At a checkpoint, let $r+2g\leq s$ with
$\widehat\sigma_{r+2g}>0$, and let $G_r$, $G_{r+g}$, and $G_{r+2g}$ be
stable reduced models. Orders for which at least one of these three
reduced models is unstable are discarded. On a fixed frequency grid
$\Omega$, define
\[
    d_1(r)
    :=
    \max_{\omega\in\Omega}
    \|G_{r+g}(\mathrm{i}\omega)-G_r(\mathrm{i}\omega)\|_2,
    \qquad
    d_2(r)
    :=
    \max_{\omega\in\Omega}
    \|G_{r+2g}(\mathrm{i}\omega)-G_{r+g}(\mathrm{i}\omega)\|_2.
\]
If $d_1(r)>0$ and
$\rho_r:=d_2(r)/d_1(r)<1$, we set
\[
    \eta_r
    :=
    s_{\rm safe}\frac{d_1(r)}{1-\rho_r},
    \qquad s_{\rm safe}\geq1.
\]
If $d_1(r)=d_2(r)=0$, we set $\eta_r=0$, and otherwise
$\eta_r=+\infty$. The provisional order is the smallest $r$ with
$\eta_r\leq\tau_{\rm bt}$.

For consecutive checkpoints with the same provisional order, let
$s_{\rm old}$ and $s_{\rm new}$ denote the numbers of available
approximate Hankel singular values and set
\[
    k_*:=\min\{r+g,s_{\rm old},s_{\rm new}\}.
\]
We monitor
\[
    \eta_{\rm hsv}
    :=
    \max_{1\le k\le k_*}
    \frac{
        |\widehat\sigma_k^{\rm new}-\widehat\sigma_k^{\rm old}|
    }{
        \max\{\widehat\sigma_k^{\rm new},
        \widehat\sigma_1^{\rm new}\varepsilon_{\rm mach}\}
    },
\]
where $\varepsilon_{\rm mach}$ denotes machine precision, and
\[
    \eta_{\rm rom}
    :=
    \frac{
        \displaystyle
        \max_{\omega\in\Omega}
        \|G_r^{\rm new}(\mathrm{i}\omega)
          -G_r^{\rm old}(\mathrm{i}\omega)\|_2
    }{
        \displaystyle
        \max\left\{
        \max_{\omega\in\Omega}
        \max\{\|G_r^{\rm new}(\mathrm{i}\omega)\|_2,
              \|G_r^{\rm old}(\mathrm{i}\omega)\|_2\},
        \varepsilon_{\rm mach}
        \right\}
    }.
\]
The iteration stops once the provisional order remains unchanged, the
reduced model is stable, and both indicators satisfy prescribed tolerances
for a specified number of consecutive checks.

For reference, let
$\mathcal P_\nu
    :=
    R^{[\nu]}(R^{[\nu]})^*$,
$\mathcal Q_\nu
    :=
    S^{[\nu]}(S^{[\nu]})^*$,
and set
$L_P=L\otimes I_m$ and $L_Q=L\otimes I_p$.
Since
\[
    K_P+K_P^\top=L_P^\top L_P,
    \qquad
    K_Q+K_Q^\top=L_Q^\top L_Q,
\]
the corresponding Lyapunov residuals factor as
\begin{align*}
    \mathcal A\mathcal P_\nu
    +\mathcal P_\nu\mathcal A^*
    +\mathcal B\mathcal B^*
    &=
    W_PW_P^*,
    &
    W_P
    &:=
    \mathcal B-R^{[\nu]}L_P^\top,\\
    \mathcal A^*\mathcal Q_\nu
    +\mathcal Q_\nu\mathcal A
    +\mathcal C^*\mathcal C
    &=
    W_QW_Q^*,
    &
    W_Q
    &:=
    \mathcal C^*-S^{[\nu]}L_Q^\top.
\end{align*}
We report the normalized nuclear residuals
\[
    \rho_P
    :=
    \frac{\|W_P\|_{\rm HS}^2}{\|B\|_F^2},
    \qquad
    \rho_Q
    :=
    \frac{\|W_Q\|_{\rm HS}^2}{\|C\|_F^2}.
\]
These quantities are used only as diagnostics and do not enter the
stopping criterion.
Finally,
\[
    \widehat\varepsilon_r
    :=
    2\sum_{k=r+1}^{s}\widehat\sigma_k
\]
is recorded as a finite-factor diagnostic. Neither $\eta_r$ nor
$\widehat\varepsilon_r$ is a rigorous error bound for the original delay
system.
\section{Multiple delays}
\label{sec:multiple-delays}

The preceding construction extends directly to finitely many discrete
delays. Consider
\begin{equation}
    \dot x(t)
    =
    A_0x(t)
    +
    \sum_{j=1}^d A_jx(t-h_j)
    +
    Bu(t),
    \qquad
    y(t)=Cx(t),
    \qquad
    0<h_1<\cdots<h_d=:h,
    \label{eq:multiple-delay-system}
\end{equation}
and assume exponential stability. On
$\mathcal X=\C^n\times L^2(-h,0;\C^n)$,
the generator is obtained from \eqref{eq:dde-generator} by replacing
$A_1\phi(-h)$ with
$\sum_{j=1}^d A_j\phi(-h_j)$.
The characteristic matrix becomes
\[
    \Delta(s)
    =
    sI-A_0-\sum_{j=1}^d A_je^{-sh_j},
    \qquad
    \Delta'(s)
    =
    I+\sum_{j=1}^d h_jA_je^{-sh_j}.
\]
Hence the characteristic-root refinement and shift construction remain
unchanged apart from replacing the single delay term by the corresponding
finite sum.

The real Takenaka--Malmquist representation is unchanged. With
$K_P=K\otimes I_m$, $L_P=L\otimes I_m$ and
$K_Q=K\otimes I_p$, the reachability coefficient matrix satisfies
\[
    A_0X
    +
    \sum_{j=1}^d A_jXe^{-h_jK_P}
    -
    XK_P
    =
    -BL_P.
\]
The adjoint history component is piecewise $H^1$, with jump conditions at
the interior delay points $-h_j$, $j<d$, and a boundary condition at
$-h$. Accordingly, the observability factor has the form
\[
    Sw
    =
    \begin{pmatrix}
        Yw\\[1mm]
        \displaystyle
        \theta\mapsto
        \sum_{j=1}^d
        \mathbf 1_{[-h_j,0]}(\theta)\,
        A_j^\top Y
        e^{-(\theta+h_j)K_Q}w
    \end{pmatrix},
\]
where
\[
    A_0^\top Y
    +
    \sum_{j=1}^d A_j^\top Ye^{-h_jK_Q}
    -
    YK_Q
    =
    -C^\top L_Q.
\]
The blockwise ADI construction of \Cref{sec:adi-real-block} carries over
by replacing the single delayed term with the corresponding finite sums.
In particular, a real shift still requires a linear system of order $n$,
whereas a conjugate pair requires a real coupled system of order $2n$.
The number of delays therefore affects the assembly of the coefficient
matrices, but not their dimensions.

The only change in the balanced-truncation postprocessing is the history
contribution to the cross product. With
$D_j:=Y^\top A_jX$,
one obtains
\[
    S^*R
    =
    Y^\top X
    +
    \sum_{j=1}^d
    \int_0^{h_j}
        e^{-tK_Q^\top}
        D_j
        e^{-(h_j-t)K_P}\,dt.
\]
Each integral is evaluated by the same block-exponential formula as in
\eqref{eq:channel-block-exponential}. Once $S^*R$ has been assembled, the
remaining balanced-truncation formulas and the stopping criteria are
unchanged.

Finally, the proof of
\Cref{prop:dde-reachability-regularity} extends verbatim, with
\[
    \Phi'(t)
    =
    A_0\Phi(t)
    +
    \sum_{j=1}^d
    \mathbf 1_{[h_j,\infty)}(t)A_j\Phi(t-h_j).
\]
Hence
$\ran\mathcal R\subset\dom(\mathcal A)$,
and \Cref{thm:adi-bt-convergence} applies to
\eqref{eq:multiple-delay-system} as well.

The implementation supplied as Online Resource~1 and the numerical
experiments below use the single-delay specialization.

\section{Numerical experiments}
\label{sec:numerics}
We assess the method in two complementary settings. First, small delay
systems are used for validation against independent Chebyshev--Lobatto
discretizations of the history-state realization. Second, a resonant
mechanical system with delayed internal coupling serves as a benchmark for
the complete procedure. All numerical experiments use the single-delay
formulation \eqref{eq:dde}.

All computations use the real block implementation of
\Cref{sec:adi-realification,sec:adi-real-block}. We additionally verified
the implementation against the complex Takenaka--Malmquist coordinates for
identical shift sequences; the resulting cross products, approximate
Hankel singular values, and reduced transfer functions agree to roundoff.

Since all system matrices are real,
$\|G(-\mathrm{i}\omega)\|_2=\|G(\mathrm{i}\omega)\|_2$, and likewise for
the transfer-function differences considered below. It therefore suffices
to sample nonnegative frequencies. For the validation examples,
transfer-function errors are evaluated on
\[
    \Omega_{\rm test}
    =
    \left\{
        10^{-3+6(j-1)/1999}
        :
        j=1,\ldots,2000
    \right\},
\]
and we write
\[
    e_\Omega(G_1,G_2)
    :=
    \max_{\omega\in\Omega_{\rm test}}
    \|G_1(\mathrm{i}\omega)-G_2(\mathrm{i}\omega)\|_2.
\]
The transfer function of the delay system is evaluated directly as
\[
    G(s)
    =
    C\bigl(sI-A_0-A_1e^{-sh}\bigr)^{-1}B.
\]
Unless stated otherwise, characteristic roots for shift generation are
localized with Chebyshev order $60$. The initial search includes characteristic roots in the half-plane
$\Re\lambda\geq-2/h$.
The search region is enlarged whenever additional shift blocks are
required. Shift blocks are processed in successive bands containing four
blocks, with each band repeated twice.

For goal-oriented order selection and successive-ROM checks we use
\[
    \Omega_{\rm stop}
    =
    \{0\}
    \cup
    \left\{
        \omega_*10^{-3+6(j-1)/23}
        :
        j=1,\ldots,24
    \right\},
    \quad
    \omega_*
    =
    \max\{1,h^{-1},\|A_0\|_2,\|A_1\|_2\}.
\]
We use $g=4$ for the nested reduced models and
$s_{\rm safe}=2$. The stopping tolerances for the approximate Hankel
singular values and successive reduced models are both $10^{-3}$, and two
consecutive successful checks are required.

\subsection{Validation and independent reference}
\label{sec:numerics-validation}

We first consider three stable single-delay systems. The scalar example is
\[
    A_0=-2,\qquad
    A_1=0.5,\qquad
    B=C=1,\qquad
    h=1.
\]
For the two-dimensional examples we use
\[
    A_0=
    \begin{bmatrix}
        -2&0.3\\
        -0.2&-1.5
    \end{bmatrix},
    \qquad
    A_1=
    \begin{bmatrix}
        -0.2&0.1\\
        0.05&-0.15
    \end{bmatrix},
    \qquad
    h=1.
\]
The SISO case has
\[
    B=
    \begin{bmatrix}1\\0.5\end{bmatrix},
    \qquad
    C=
    \begin{bmatrix}1&0.3\end{bmatrix},
\]
whereas the MIMO case uses
\[
    B=
    \begin{bmatrix}
        1&0.2\\
        0.3&1
    \end{bmatrix},
    \qquad
    C=
    \begin{bmatrix}
        1&0.3\\
        0.2&1
    \end{bmatrix}.
\]
For these tests we use the absolute order-selection tolerance
$\tau_{\rm bt}=10^{-3}$. The automatically selected orders and sampled
transfer errors are reported in \Cref{tab:regression-tests}.

\begin{table}[t]
\centering
\caption{Low-dimensional validation examples. Here $J$ denotes the number
of scalar ADI shifts and $\alpha(A_r)=\max\Re\sigma(A_r)$.}
\label{tab:regression-tests}
\begin{tabular}{lrrrr}
\toprule
example
& $r$
& $J$
& $e_\Omega(G,G_r)$
& $\alpha(A_r)$
\\
\midrule
scalar SISO
& 12
& 78
& $5.676\cdot10^{-4}$
& $-4.465\cdot10^{-1}$
\\
$n=2$ SISO
& 11
& 96
& $1.849\cdot10^{-4}$
& $-5.266\cdot10^{-1}$
\\
$n=2$ MIMO
& 15
& 80
& $4.212\cdot10^{-4}$
& $-6.568\cdot10^{-1}$
\\
\bottomrule
\end{tabular}
\end{table}

For an independent validation of the operator products entering balanced
truncation, we compare the MIMO example with Chebyshev--Lobatto
discretizations of the history transport equation; see
\cite{BredaMasetVermiglio2005,MichielsJarlebringMeerbergen2011}.
A collocation order $N$ produces a finite-dimensional realization of
dimension $2(N+1)$. We compute its dense Gramians and apply square-root
balanced truncation with the same order $r=15$ as for the operator-ADI
model. Denoting the collocation transfer function by $G_N$ and its
order-$15$ balanced truncation by $G_{r,N}$ gives \Cref{tab:chebyshev-reference}.

\begin{table}[t]
\centering
\caption{Comparison with Chebyshev--Lobatto reference models for the
two-dimensional MIMO example.}
\label{tab:chebyshev-reference}
\begin{tabular}{rrrrrr}
\toprule
$N$
& dim.
& $\alpha(A_N)$
& $e_\Omega(G,G_N)$
& $e_\Omega(G,G_{r,N})$
& $e_\Omega(G_r,G_{r,N})$
\\
\midrule
20
& 42
& $-1.827$
& $1.117\cdot10^{-4}$
& $4.212\cdot10^{-4}$
& $6.979\cdot10^{-6}$
\\
40
& 82
& $-1.827$
& $3.244\cdot10^{-5}$
& $4.212\cdot10^{-4}$
& $5.092\cdot10^{-8}$
\\
60
& 122
& $-1.827$
& $1.611\cdot10^{-5}$
& $4.212\cdot10^{-4}$
& $6.365\cdot10^{-8}$
\\
\bottomrule
\end{tabular}
\end{table}

Thus the operator-ADI reduced model and the balanced truncations of the two
finest spectral realizations agree to approximately $10^{-8}$ on
$\Omega_{\rm test}$. The agreement also holds for the approximate Hankel
singular values: for the first $20$ values, the relative difference between
the operator-ADI computation and the $N=60$ Chebyshev reference is below
$4\cdot10^{-5}$.

\subsection{A resonant mechanical delay benchmark}
\label{sec:numerics-mechanical}

We next consider a mechanical system with eight degrees of freedom and
delayed internal coupling,
\begin{equation}
    M\ddot q(t)
    +D\dot q(t)
    +Kq(t)
    +D_d\dot q(t-h)
    +K_dq(t-h)
    =
    B_u u(t),
    \qquad
    y(t)=C_q q(t),
    \label{eq:mechanical-benchmark}
\end{equation}
where $q(t)\in\R^8$ and $h=1$. The mass matrix is
\[
    M
    =
    \diag(1,1.15,0.9,1.1,0.95,1.05,1.2,0.85).
\]
The stiffness matrix $K$ is that of a chain with boundary and
interconnection spring constants
\[
    (k_0,\ldots,k_8)
    =
    (45,32,55,38,60,35,52,41,58),
\]
that is,
$K_{jj}=k_{j-1}+k_j$,
$K_{j,j+1}=K_{j+1,j}=-k_j$.
The instantaneous damping is
$D=0.0675M+10^{-3}K$.
Writing $e_j$ for the $j$th canonical vector in $\R^8$, set
$e_a=e_2-e_7$,
$e_b=e_3-e_6$.
The delayed stiffness and damping couplings are
\[
    K_d
    =
    6\bigl(e_ae_a^\top+0.7e_be_b^\top\bigr),
    \qquad
    D_d
    =
    0.002\bigl(e_ae_a^\top+0.5e_be_b^\top\bigr).
\]
Two forces act at the outer masses and two interior displacements are
observed:
\[
    B_u=[\,e_1\;\;e_8\,],
    \qquad
    C_q=
    \begin{bmatrix}
        e_3^\top\\
        e_6^\top
    \end{bmatrix}.
\]

With $x=(q^\top,\dot q^\top)^\top$,
\eqref{eq:mechanical-benchmark} has the form \eqref{eq:dde} with
instantaneous state dimension $n=16$ and
\begin{align*}
    A_0
    &=
    \begin{bmatrix}
        0&I\\
        -M^{-1}K&-M^{-1}D
    \end{bmatrix},
    &
    A_1
    &=
    \begin{bmatrix}
        0&0\\
        -M^{-1}K_d&-M^{-1}D_d
    \end{bmatrix},
\\
    B
    &=
    \begin{bmatrix}
        0\\M^{-1}B_u
    \end{bmatrix},
    &
    C&=
    \begin{bmatrix}
        C_q&0
    \end{bmatrix}.
\end{align*}

The parameters are chosen so that the delayed coupling has a pronounced
effect on the resonant input--output dynamics. In particular, the delay
does not merely perturb individual resonance frequencies: compared with
the system obtained by removing the delayed term, it substantially
enriches and reshapes the resonance pattern.

As an independent numerical stability check, we use Chebyshev--Lobatto
discretizations of the history generator. For collocation orders
$N=40,60,80$, the computed spectral abscissae agree to the displayed
precision and give
\[
    \alpha_{\rm DDE}
    \approx
    -1.0643824762\cdot10^{-2}.
\]
To illustrate the effect of the delay, we compare the delay-system
transfer function with
\[
    G_0(s)
    :=
    C(sI-A_0)^{-1}B,
\]
which is obtained by removing the delayed term while leaving the
instantaneous mechanical system unchanged. The left panel of
\Cref{fig:mechanical-benchmark} shows that the two frequency responses
have markedly different resonance structures. The order-$19$ reduced
model, on the other hand, follows the delay-system response closely over the
entire displayed frequency range.

For this benchmark we use
$\tau_{\rm bt}=5\cdot10^{-3}$.
The algorithm terminates after
$J=96$
scalar ADI shifts, corresponding to $48$ real block steps, and selects
$r=19$.
In particular,
$r=19>16=n$.
This is not a failure of model reduction. Balanced truncation is applied
to the infinite-dimensional state space
$\C^{16}\times L^2(-1,0;\C^{16})$,
rather than to the $16$-dimensional instantaneous state alone. The
dimension of the delay-free surrogate is therefore not constrained by
$n$.

The reduced model is stable, with
\[
    \alpha(A_r)
    =
    -1.077\cdot10^{-2}.
\]
For the final error evaluation we use the denser logarithmic grid
\[
    \Omega_{\rm mech}
    =
    \left\{
        10^{-3+6(j-1)/11999}
        :
        j=1,\ldots,12000
    \right\}.
\]
This gives
\[
    e_{\rm mech}
    :=
    \max_{\omega\in\Omega_{\rm mech}}
    \|G(\mathrm{i}\omega)-G_r(\mathrm{i}\omega)\|_2
    =
    2.421182\cdot10^{-3},
\]
with the maximum attained at the low-frequency end of the grid,
$\omega=10^{-3}$. Relative to the maximum sampled gain of the delay
system,
\[
    \frac{e_{\rm mech}}
    {\displaystyle
     \max_{\omega\in\Omega_{\rm mech}}
     \|G(\mathrm{i}\omega)\|_2}
    =
    2.27372\cdot10^{-3},
\]
that is, approximately $0.227\%$.

At termination, the goal-oriented ROM-tail indicator and the finite-factor
balanced-trun\-cation tail diagnostic are
\[
    \eta_r
    =
    4.851\cdot10^{-3},
    \qquad
    \widehat\varepsilon_r
    =
    2.886\cdot10^{-3},
\]
whereas the global ADI residual indicators are
\[
    \rho_P
    =
    5.690\cdot10^{-3},
    \qquad
    \rho_Q
    =
    2.223\cdot10^{-5}.
\]
As discussed in \Cref{sec:stopping},
$\eta_r$ and $\widehat\varepsilon_r$ are diagnostics rather than rigorous
a posteriori error bounds for the original delay system. The right panel
of \Cref{fig:mechanical-benchmark} shows the frequency-dependent error
together with the finite-factor balanced-truncation tail diagnostic
$\widehat\varepsilon_r$.

\begin{figure}[t]
\centering
\definecolor{fullblue}{RGB}{0,114,178}
\definecolor{nodelayorange}{RGB}{213,94,0}
\definecolor{romgreen}{RGB}{0,158,115}

\begin{tikzpicture}
\begin{groupplot}[
    group style={
        group size=2 by 1,
        horizontal sep=1.15cm,
    },
    width=0.5\textwidth,
    height=5.35cm,
    ymode=log,
    xlabel={$\omega$},
    xlabel style={font=\small},
    ylabel style={font=\small},
    tick label style={font=\scriptsize},
    title style={font=\small},
    grid=both,
    major grid style={draw=gray!24},
    minor grid style={draw=gray!11},
    axis line style={line width=0.55pt},
    tick style={line width=0.45pt},
]

\nextgroupplot[
    title={(a) Frequency responses},
    xmin=0,
    xmax=16,
    xtick={0,4,8,12,16},
    ymin=3e-4,
    ymax=2,
    ylabel={$\|G(\mathrm{i}\omega)\|_2$},
    legend style={
        at={(0.03,0.03)},
        anchor=south west,
        font=\scriptsize,
        draw=black!35,
        fill=white,
        fill opacity=0.90,
        text opacity=1,
        inner sep=2pt,
        row sep=-1pt,
    },
]

\addplot[
    color=fullblue,
    solid,
    line width=1.35pt,
    mark=none
]
table[x=omega,y=full] {mechanical_frequency.dat};
\addlegendentry{full delay system}

\addplot[
    color=nodelayorange,
    dashed,
    line width=1.35pt,
    mark=none
]
table[x=omega,y=nodelay] {mechanical_frequency.dat};
\addlegendentry{without delay term}

\addplot[
    color=romgreen,
    dashdotted,
    line width=1.20pt,
    mark=none
]
table[x=omega,y=rom] {mechanical_frequency.dat};
\addlegendentry{reduced model, $r=19$}

\nextgroupplot[
    title={(b) Reduction error},
    xmin=0,
    xmax=30,
    xtick={0,10,20,30},
    ymin=1e-4,
    ymax=5e-3,
    ylabel={$\|G-G_r\|_2$},
    ylabel style={
        at={(axis description cs:1.12,0.5)},
        anchor=south
    },
]

\addplot[
    color=fullblue,
    solid,
    line width=1.35pt,
    mark=none
]
table[x=omega,y=error] {mechanical_error.dat};

\addplot[
    color=romgreen,
    densely dotted,
    line width=1.20pt,
    mark=none
]
coordinates {(0,2.886e-3) (30,2.886e-3)};

\node[
    anchor=south east,
    font=\scriptsize,
    text=romgreen,
    fill=white,
    fill opacity=0.85,
    text opacity=1,
    inner sep=1pt
] at (axis cs:29.2,2.886e-3)
{$\widehat{\varepsilon}_r$};

\end{groupplot}
\end{tikzpicture}
\caption{Mechanical delay benchmark.
Left: frequency responses of the delay system, the system obtained
by removing the delayed term, and the order-$19$ reduced model.
Right: frequency-dependent reduction error; the dotted line indicates
the finite-factor balanced-truncation tail diagnostic
$\widehat{\varepsilon}_r$.}
\label{fig:mechanical-benchmark}
\end{figure}
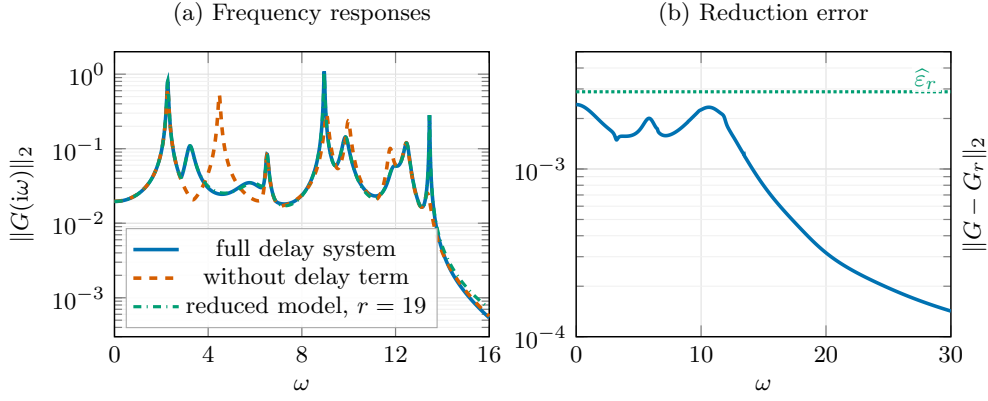

The convergence history further illustrates the distinction between
goal-oriented stopping and reduction of the global LR-ADI residuals. The provisional
order has already settled at $r=19$ by $J=48$. Selected quantities at the
subsequent checkpoints are reported in
\Cref{tab:mechanical-convergence}.

\begin{table}[t]
\centering
\caption{Convergence history for the mechanical benchmark. A dash indicates
that the stabilization indicator is not defined because the provisional
order changed at the corresponding checkpoint.}
\label{tab:mechanical-convergence}
\begin{tabular}{rrrrrr}
\toprule
$J$
& $\rho_P$
& $\rho_Q$
& $\eta_{\rm hsv}$
& $\eta_{\rm rom}$
& $r$
\\
\midrule
48
& $3.010\cdot10^{-2}$
& $2.127\cdot10^{-3}$
& --
& --
& 19
\\
64
& $1.141\cdot10^{-2}$
& $1.720\cdot10^{-4}$
& $3.391\cdot10^{-2}$
& $2.152\cdot10^{-7}$
& 19
\\
80
& $7.558\cdot10^{-3}$
& $5.028\cdot10^{-5}$
& $9.760\cdot10^{-8}$
& $4.215\cdot10^{-13}$
& 19
\\
96
& $5.690\cdot10^{-3}$
& $2.223\cdot10^{-5}$
& $1.109\cdot10^{-10}$
& $2.718\cdot10^{-14}$
& 19
\\
\bottomrule
\end{tabular}
\end{table}

Thus the reduced input--output model has effectively stabilized while the
global reachability residual is still several orders of magnitude larger
than the reduced-model change. At $J=80$, both stabilization indicators
already satisfy their tolerances; the iteration terminates at $J=96$ after
the required second successful checkpoint.

The complete MATLAB call for this benchmark, including characteristic-root
localization, the ADI iteration, stopping and order selection, and the final
balanced-truncation postprocessing, required $1.73\,\mathrm{s}$
on a Windows 64-bit system with an Intel Core Ultra 7 258V processor and
$32\,\mathrm{GB}$ RAM, using MATLAB R2026a Update~2.

\section{Conclusions}
\label{sec:conclusions}

We have developed an operator-ADI approach to balanced truncation of delay
differential systems that works directly with the infinite-dimensional
history-state realization while reducing the required computations to
finite-dimensional linear algebra. Under suitable assumptions, the
resulting reduced transfer functions converge in $\mathcal H_\infty$ to
the exact infinite-dimensional balanced truncation.

The construction extends directly to finitely many discrete delays without
increasing the dimensions of the shifted linear systems. Extensions to more
adaptive shift strategies and to distributed or neutral delays remain
topics for future work.

\backmatter

\bmhead*{Supplementary information}
The MATLAB implementation used for the numerical experiments is provided
as Online Resource~1.

\bmhead*{Statements and Declarations}

\bmhead*{Conflict of Interest}
The author declares that he has no conflict of interest.

\bmhead*{Funding}
The author received no external funding for this research.

\bmhead*{Author contributions}
The author conceived the study, developed the methodology, performed the analysis and numerical experiments, wrote the software, and prepared and revised the manuscript.

\bmhead*{Acknowledgement}
Not applicable.

\bmhead*{Data availability}
No external datasets were used. All numerical data reported in this article
are generated by the MATLAB implementation provided as Online Resource~1.

\bmhead*{Code availability}
The MATLAB implementation used to generate the numerical results is provided
as Online Resource~1.

\bmhead*{Use of generative AI}
During the preparation of this manuscript, the author used ChatGPT
(OpenAI) for limited assistance with language editing and as a coding
assistant for parts of the MATLAB implementation. All mathematical ideas,
methods, arguments, and proofs are the author's own. All AI-assisted text
and code were critically reviewed and, where necessary, revised by the
author, who takes full responsibility for the final content.

\bibliography{references}

\end{document}